\documentclass[11pt]{article}
\usepackage{geometry}
\usepackage{subcaption}
\usepackage{amsthm,amsmath,amssymb,amsfonts,hyperref}
\usepackage{epsfig,color}
\usepackage{enumerate}

\usepackage{arydshln}	
\usepackage{pdflscape} 
\usepackage{multirow}
\usepackage{rotating} 
\usepackage{ulem} 

\newtheorem{thm}{Theorem}[section]
\newtheorem{defn}[thm]{Definition}
\newtheorem{prop}[thm]{Proposition}
\newtheorem{cor}[thm]{Corollary}
\newtheorem{rmk}[thm]{Remark}
\newtheorem{lma}[thm]{Lemma}

\newtheorem{exm}[thm]{Example}

\newcommand{\eq}{\begin{equation}}
\newcommand{\qe}{\end{equation}}

\def\N{{\rm I\kern-0.16em N}}
\def\R{{\rm I\kern-0.16em R}}
\def\E{\mathbb{E}}
\def\P{{\rm I\kern-0.16em P}}
\def\F{{\rm I\kern-0.16em F}}
\def\B{{\rm I\kern-0.16em B}}
\def\C{{\rm I\kern-0.46em C}}
\def\G{{\rm I\kern-0.50em G}}
\def\Z{{\mathbb{Z}}}

\numberwithin{equation}{section}

\begin{document}

\normalem 

\title{On infinite covariance expansions}

\date{}

%
%
%

\renewcommand{\thefootnote}{\fnsymbol{footnote}}
\begin{small}
  \author{Marie Ernst\footnotemark[1], \, 
Gesine Reinert\footnotemark[2] \, and Yvik Swan\footnotemark[1]}
\end{small}
\footnotetext[1]{Universit\'e de Li\`ege, corresponding author Yvik
  Swan: yswan@uliege.be.}
\footnotetext[2]{University of Oxford.}

\maketitle

\begin{abstract}
  In this paper we provide a probabilistic representation of
  Lagrange's identity which we use to obtain Papathanasiou-type
  variance expansions of arbitrary order. Our expansions lead to
  generalized sequences of weights which depend on an arbitrarily
  chosen sequence of (non-decreasing) test functions.  The expansions
  hold for arbitrary univariate target distribution under 
  weak assumptions, in particular they hold for continuous and
  discrete distributions alike.  The weights are studied under
  different sets of assumptions either on the test functions or on the
  underlying distributions. Many concrete illustrations for standard
  probability distributions are provided (including Pearson, Ord,
  Laplace, Rayleigh, Cauchy, and Levy distributions).

  \medskip
  
\end{abstract}


\section{Introduction}
\label{sec:introduction}


The starting point of this paper is the famous Gaussian expansion
which states that if $N \sim \mathcal{N}(0, 1)$, then
  \begin{align}
&    \mathrm{Var}[g(N)] = \sum_{k=1}^{\infty} \frac{(-1)^{k+1}}{k!}
                                                       \mathbb{E}\left[g^{(k)}(N)^2\right]
\label{eq:104}
  \end{align}
  for all smooth functions $g:\R\to\R$ such that all the expectations
  exist. Expansion \eqref{eq:104}{, whose first order term yields
    an upper variance bound generalizing Chernoff's famous Gaussian
    bound from \cite{C81}}, has been obtained in a number of different
  (and often non equivalent) ways. It is proved in \cite{HK95} via
  orthogonality properties of Hermite polynomials, and extensions to
  multivariate and infinite dimensional settings are given in
  \cite{HP95,houdre1998interpolation}. Chen uses martingale and
  stochastic integrals to obtain a
   general version of
  \eqref{eq:104} (also valid on certain manifolds) in
  \cite{chen1985poincare}.  The expansion is contextualized in
  \cite{Le95} through properties of the Ornstein-Uhlenbeck operator,
  and it is also shown in that paper that the semi-group arguments
  carry through to non-Gaussian target distributions under 
  general assumptions.  A very general approach to this line of
  research can be found in \cite{houdre1998interpolation} where
  similar expansions are obtained by means of an iteration of an
  interpolation formula for infinitely divisible distributions. The
  main difference between the univariate standard Gaussian and the
  general non-Gaussian target is that the explicit weight sequence and
  simple iterated derivatives appearing in \eqref{eq:104} need to be
  replaced by some well-chosen iterated gradients with weight
  sequences which can be quite difficult to obtain explicitly (for
  instance Ledoux' sequence from \cite{Le95} is an iteration of the
  ``carr\'e du champ'' operator).

  The above references are predated by
  \cite{papathanasiou1988variance} wherein a general version of
  \eqref{eq:104} (valid for arbitrary continuous target distributions)
  is obtained through elementary arguments relying on an iteration of the
  exact Cauchy-Schwarz equality (via the so-called \emph{Mohr and Noll
    identity} from \cite{mohr1952bemerkung}) combined with the
  Lagrange identity for integrals due to \cite{CaPa85}.
  Papathanasiou's method of proof is extended in \cite{APP07} to
  encompass discrete distributions. Both the continuous and discrete
  expansions are of the same form as \eqref{eq:104}, although the
  weight sequence $(-1)^k/k!$ is replaced with a target-specific
  explicit sequence of weights (see equations \eqref{eq:101} and
  \eqref{eq:4} below).  To set the scene, we borrow notation from
  \cite{ERS19vb1} which allows to unify the presentation of the
  results from \cite{papathanasiou1988variance} and \cite{APP07} and shall be
  used throughout this paper.

  \medskip
  
  \noindent \textbf{Notation:} For a function $f: \R \rightarrow \R$
  let $\Delta^{\ell}f(x) = (f(x+\ell) - f(x))/\ell$ for all
  $\ell \in \{-1,0,1\}$, with the convention that
  $\Delta^0f(x) = f'(x)$, with $f'(x)$ the weak derivative defined
  Lebesgue almost everywhere.  The case $\ell=0$ is referred to as the
  \emph{continuous case} and $\ell \in \left\{ -1,1 \right\}$ is
  referred to as the \emph{discrete case}.  
  For a real-valued function $f$, in the continuous case $f^{(k)}$
  denotes its $k^{th}$ derivative; discrete higher order derivatives
  $f^{(k)}$ are obtained by iterating the forward derivative
  $\Delta^+f(x) = f(x+1) - f(x)$.  We use the rising and falling
  factorial notation
  \begin{align}
    \label{eq:discrprod1}
    &    f^{[k]}(x)
       =  \prod_{j=0}^{k-1} f(x+j)
\mbox{ and }
      f_{[k]}(x) 
      = \prod_{j=0}^{k-1} f(x-j)
,  
  \end{align}
  with the convention that $f^{[0]}(x)= f_{[0]}(x) =1$.

  \medskip 
    
 Expansion
  \eqref{eq:104} can then be seen as a particular instance of the
  following result (see   \cite[Theorem 1 and Corollary 1]{papathanasiou1988variance}  and
  \cite[Theorem~3.1]{APP07}).
  
  \begin{thm}[Papathanasiou's expansion]\label{thm:papaclass} 
    Let $X$ be a random variable with finite {$(n+2)^{th}$}
    moments. Let $g$ be a real-valued function with finite variance
    with respect to $X$. Then
   \begin{equation}
    \label{eq:100}
    \mathrm{Var}[g(X)] = \sum_{k=1}^{n} (-1)^{k-1}
    \mathbb{E} \left[ (g^{(k)}(X))^2 \Gamma_{k}(X) \right] + (-1)^{n}R_n
  \end{equation}
  where $R_n$ is a {non-negative} remainder term 
  and 
  $\Gamma_k$ depend on  the type of distribution, as follows. 
  \begin{enumerate}
  \item  If $X$ is a real random variable with continuous probability
    density function (pdf) $p$, then 
   the  weights are
  \begin{align}
    \label{eq:101}
    \Gamma_{k}(t) = \frac{(-1)^{k-1}}{k!(k-1)!p(t)}\bigg( 
  \mathbb{E}\left[(X-t)^{k}\right] \int_{-\infty}^t 
    (x-t)^{k-1} p(x) dx - \mathbb{E}\left[(X-t)^{k-1}\right] \int_{-\infty}^t
      (x-t)^{k}p(x) dx\bigg),
  \end{align}
  defined for all $t$ such that $p(t)>0$. 

 \item If $X$ is an integer-valued r.v.\ with probability mass
   function (pmf) $p$, then 
   the weights are
  \begin{align}
    \label{eq:4}
    \Gamma_{k}(t) &= \frac{(-1)^{k-1}}{k!(k-1)!p(t)} \bigg(
    \mathbb{E}\left[(X-t)_{[k]}\right] \sum_{x<t+1} p(x) (x-(t+1))_{[k-1]} \nonumber\\
  &  \quad -  \mathbb{E}\left[(X-(t+1))_{[k-1]}\right] \sum_{x<t} p(x) (x-t)_{[k]} \bigg)  , 
  \end{align}
  defined for all $t$ such that $p(t)>0$. 

  \end{enumerate} 
\end{thm}

It is not hard to show that when $X \sim \mathcal{N}(0,1)$, the weight
sequence \eqref{eq:101} simplifies to {$\Gamma_k(t) = {1}/{k!}$}
so that \eqref{eq:100} indeed contains \eqref{eq:104}. More generally,
it is shown in \cite{johnson1993note} that if $p$ belongs to the
Integrated Pearson (IP) system of distributions (see Definition
\ref{defn:lagr-ident-infin-1}) then the weights take on a particularly
agreeable form, namely
$ \Gamma_{k}(t) = {\Gamma_1(x)^k}/{(k! \prod_{j=0}^k(1 - j \delta))}$
and $\delta =\Gamma_1''(x)$ (which is constant if $X$ is Integrated
Pearson); many familiar univariate distributions belong to the IP
system, {such as} the normal, beta, gamma, {and} Student
{distributions}.
Similarly as in the continuous case, it is shown by \cite[Corollary
4.1]{APP07} that if $X$ belongs to the {cumulative Ord family} with
parameter $(\delta, \beta, \gamma)$ defined in Definition
\ref{defn:lagr-ident-infin-1ORD}, then the weights in \eqref{eq:4} are
 $\Gamma_{k}(t) = \Gamma_1^{[k]}(t)/\big(k! \prod_{j=0}^k(1 - j
  \delta)\big)$.
  Like its continuous counterpart, the discrete IP
system also contains many familiar univariate distributions such as
the binomial, Poisson and geometric distributions.


The list of references presented so far is anything but exhaustive
and 
expansions inspired from \eqref{eq:104} have attracted a lot of
attention over the years, e.g.\ with extensions to matrix inequalities
as in \cite{olkin2005matrix,wei2009covariance,afendras2011matrix}, to
stable distributions \cite{koldobsky1996inequalities}, to Bernoulli
random vectors \cite{BGH01}; more references shall be provided in the
text.  Aside from their intrinsic interest, 
they have many applications and are closely connected to a wide
variety of profound mathematical questions. For statistical inference
purposes, they can be used in the study of the variance of classes of
estimators (see e.g.\ \cite[section 5]{APP07}), of copulas
(\cite{cuadras2008eigenanalysis}), for problems related to
superconcentration (\cite{chatterjee2014superconcentration} and 
  \cite{tanguy2017quelques}) or for the study of correlation inequalities
\cite{houdre1998interpolation} and \cite{blazquez2014maximal}. These expansions
can also interpreted as refined log-Sobolev, Poincar\'e or
isoperimetric inequalities, see \cite{saumard2018weighted}. The
weights appearing in the first order ($n=1$) bounds are crucial
quantities in Stein's method \cite{fathi2018stein,ledoux2015stein} and
their higher order extensions are closely connected to eigenvalues and
eigenfunctions of certain differential operators
\cite{chen1985poincare}. 


\medskip

In the present paper, we combine the method from
\cite{papathanasiou1988variance,APP07} with intuition from \cite{K85}
{(and our recent work \cite{ERS19vb1})} to unify and extend the results
from Theorem \ref{thm:papaclass}
to arbitrary targets under {very weak}
assumptions. The result is
given in Theorem~\ref{thm:var-} and can be briefly sketched in a
simplified form as follows.  Fix $(\ell_k)_{k\ge1}$ a sequence either
in $\left\{ -1, 1 \right\}$ or $\left\{ 0 \right\}$ and let
$h: \R \to \R$ be such that {${\Delta^{-\ell_i} h}\ge0$ for all
  $i\ge 1$. Starting with some functions $f, g: \R \to \R$, we
  recursively define the sequence $({f}_k)_{k\ge 0}$ (resp.,
  $({g}_k)_{k\ge 0}$) by ${f}_0(x)={f}(x)$ (resp., ${g}_0(x)={f}(x)$)
  and ${f}_{i}(x)={\Delta^{-\ell}{f}_{i-1}(x)}/{\Delta^{-\ell}h(x)}$
  (resp.,
  ${g}_{i}(x)={\Delta^{-\ell}{g}_{i-1}(x)}/{\Delta^{-\ell}h(x)}$) for
  all $x \in \mathcal{S}(p)$. Then, for all $n \ge 1$, it holds that
if the expectations below are finite then}
\begin{equation}\label{eq:3}
    \mathrm{Cov} \left[f(X), g(X)  \right]
          = 
 \sum_{k=1}^n (-1)^{k-1} \E\left[ 
                      \Delta^{-\ell_k} {f}_{k-1}(X)
                                    \Delta^{-\ell_k} {g}_{k-1}(X) 
 \frac{\Gamma_k^{
                    \ell}({h})(X) }{\Delta^{-\ell}
                        h(X)} 
                               \right] 
   + (-1)^n R_n^{\ell}({h})
\end{equation}
where the weight sequences $\Gamma_k^{ \ell}({h})$ as well as the
non-negative remainder term $R_n^{\ell}({h})$ are given explicitly
(see Theorem \ref{thm:var-}) and in many cases {{have a simple form}}
(see Section \ref{sec:weights-expansions}).  The expansions from
Theorem~\ref{thm:papaclass} are recovered by setting $f=g$, and
$h(x) = \mathrm{Id}(x)$ (the identity function) and, in the discrete
case, $\ell = -1$.
Far from obscuring {the} message, expansion \eqref{eq:3}, and its more
general form provided in Theorem \ref{thm:var-}, shed new light on the
expansion \eqref{eq:100} and its available extensions
 by 
 bringing a new interpretation to the weight sequences in terms of
 explicit iterated integrals and sums. This is the topic of
 Section \ref{sec:weights-expansions}. Our results also inscribe the
 topic within a context which is familiar to practitioners of the
 famous Stein's method.
This last connection nevertheless remains slightly mysterious and will
be studied in detail in future contributions.

\medskip

  
{The paper is organised as follows.} In Section \ref{sec:lagr-ident-infin} we provide the main results in
their most abstract form. After setting up the notations (inherited
mainly from 
\cite{ERS19vb1}), Section \ref{lma:var1} contains
the crucial Lagrange identity (Lemma \ref{lma:matrixcauchyschwarz})
and Section \ref{sec:main-result} contains the Papathanassiou-type
expansion (Theorem \ref{thm:var-}).  In Section
\ref{sec:weights-expansions} we provide illustrations by rewriting the
weights appearing in Theorem \ref{thm:var-} under different sets of assumptions. First,
in Section~\ref{sec:gener-cons-cont} we consider a general weighting
function $h$; next, in Section \ref{sec:some-gener-cons} we choose
certain specific intuitively attractive $h$-functions (namely the
identity, the cdf and the score); finally in Section
\ref{sec:illustrations-1} we obtain explicit expressions for various
illustrative distributions (here in particular  the connection with
existing literature on the topic is also made). For the sake or
readability,  all proofs are relegated to an Appendix.

\section{Infinite matrix-covariance expansions}
\label{sec:lagr-ident-infin}

We begin this paper by recalling some elements of the setup from our
paper \cite{ERS19vb1}.
Let $\mathcal{X} \subset \R$ and equip it with some $\sigma$-algebra
$\mathcal{A}$ and $\sigma$-finite measure $\mu$. Let $X$ be a random
variable on $\mathcal{X}$, with probability measure $P^X$ which is
absolutely continuous with respect to $\mu$; we denote $p$ the
corresponding probability density, and its support by
$\mathcal{S}(p) = \left\{ x \in \mathcal{X} : p(x)>0\right\}$.  As
usual, $L^1(p)$ is the collection of all real valued functions $f$
such that $\mathbb{E}|f(X)| < \infty$. Although we could in principle
keep the discussion to come very general, in order to make the paper
more concrete and readable in the sequel we shall restrict our
attention to distributions satisfying the following Assumption.

\

\noindent \textbf{Assumption A.} The measure $\mu$ is either the
counting measure on $\mathcal{X} = \Z$ or the Lebesgue measure on
$\mathcal{X} = \R$. If $\mu$ is the counting measure then there exist
$a, b \in \Z \cup \left\{-\infty, \infty \right\}$ such that
$\mathcal{S}(p) = [a, b]\cap \N$.  If $\mu$ is the Lebesgue measure
then there exist
$a, b \in \Z \cup \left\{-\infty, \infty \right\}$ such  that
$\overline{\mathcal{S}(p)} = [a, b]$.

\

    

 We denote $\mathrm{dom}(\Delta^{\ell})$
    the collection of functions $f : \R \to \R$ such that
    $\Delta^{\ell}f(x)$ exists and is finite {$\mu$-almost surely on
      $\mathcal{X}$}. {If $\ell=0$, this corresponds to all absolutely
      continuous functions; if $\ell = \pm1$ the domain is the
      collection of all functions on $\Z$}.  Let
    $\ell \in \left\{ -1, 0, 1 \right\}$.  Still following
    \cite{ERS19vb1} we also define
 \begin{equation}
   \label{eq:6}
   a_{\ell} = \mathbb{I}[\ell = 1] \mbox{ and } b_{\ell} = a_{-\ell} =
   \mathbb{I}[\ell = -1]
 \end{equation}
as well as 
 the  generalized indicator function
\begin{equation}
   \label{eq:kerneldeff}
   \chi^{\ell}(x, y) = \mathbb{I}{[x \le y - a_{\ell}]}
 \end{equation}
{which is defined with the obvious strict inequalities also for $x = -\infty$ and $y = \infty$,} and
\begin{align}
  \label{eq:29}
  \Phi^{\ell}_p(u, x, v) = \chi^{\ell}(u, x)\chi^{-\ell}(x, v)/p(x)
\end{align}
for all $u, v \in \mathcal{S}(p)$ (note that
$ \Phi^{\ell}_p(u, x, v)=0$ for $u > v$). The following result is
immediate but useful:
\begin{lma}
 \label{lem:chirules} For all $x, y$, it holds that 
  $\chi^{\ell^2}(x,y) + \chi^{\ell^2}(y,x) = 1 + \mathbb{I}[\ell
  =0]\mathbb{I}[x=y] - \mathbb{I}[\ell \ne 0] \mathbb{I}[x=y].$
  Moreover,
\begin{align}\label{eq:76}
\chi^\ell (u,y)  \chi^\ell (v,y) =  \chi^\ell ({\max(u,v)},y)  \mbox{ and } 
 \chi^\ell (x,u) \chi^\ell (x,v)  =  \chi^\ell (x,{\min(u,v)}).
\end{align}

\end{lma}

We conclude with another result from 
\cite{ERS19vb1}; {this results motivates the covariance expansion in Theorem \ref{thm:var-}}.

\begin{lma}
  If $f \in \mathrm{dom}(\Delta^{-\ell})$ is such that
  $\Delta^{-\ell}f$ is integrable on $[x_1, x_2] \cap \mathcal{S}(p)$
  then,
\begin{equation} 
  \label{eq:28}
f(x_2)-f(x_1) =   \mathbb{E}\left[   \Phi^{\ell}_p(x_1, X, x_2) \Delta^{-\ell}f(X)\right].
\end{equation}
 If, furthermore,  $f \in
L^1(p)$ then
\begin{equation*}
  \mathbb{E}[(f(X_2) - f(X_1)) \mathbb{I}[X_1<X_2]] = \mathbb{E}\left[
    \Phi^{\ell}_p(X_1, X, X_2) \Delta^{-\ell}f(X)\right]. 
\end{equation*}
\end{lma}

\subsection{A probabilistic Lagrange inequality}

{The first ingredient for our results is the following covariance representation} 
(recall that all proofs are
in the Appendix).

\begin{lma}\label{lma:var1}
  Let $X \sim p$ with support $\mathcal{S}(p)$.  If $X_1, X_2$ are independent copies of $X$ then
\begin{align} \label{eq:27}
\mathrm{Cov}[f(X), g(X)] & = \E\left[ \big(f(X_2)-f(X_1) \big)
  \big(g(X_2)-g(X_1) \big) \mathbb{I}{[X_1<X_2]} \right] \\
& = \frac{1}{2} \E\left[ \big(f(X_2)-f(X_1) \big)
  \big(g(X_2)-g(X_1) \big)\right]\label{eq:27bis}
\end{align}
for all  $f,g \in L^2(p)$.
\end{lma}


A {simple} representation such as \eqref{eq:27} is {obviously}
not new, per se; see e.g.\ {the variance expression in} \cite[page
122]{miclo2008mathematiques}.  In fact, treating the discrete and
continuous cases separately, one could also obtain identity
\eqref{eq:27} as a direct application of Lagrange's identity (a.k.a.\
the Cauchy-Schwarz inequality with remainder) which reads, in the
finite discrete case, as
\begin{equation}
  \left(\sum_{k=u}^v a_k^2 \right)\left(\sum_{k=u}^v b_k^2 \right) -
  \left(\sum_{k=u}^v a_k b_k \right)^2 =  \sum_{i=u}^{v-1}\sum_{j=i+1}^v (a_ib_j-a_jb_i)^2. 
\label{eq:lagrange}
\end{equation}
Using $a_k=g(k)\sqrt{p(k)}$ and $b_k=\sqrt{p(k)}$ for $k=0,\ldots,n$,
identity \eqref{eq:27} follows in the finite case.  Identity
\eqref{eq:lagrange} and its continuous counterpart will play a crucial
role in the sequel. As it turns out, they are more suited to our
cause under the following form.


\begin{lma}[A probabilistic Lagrange identity]\label{lma:matrixcauchyschwarz}
  Fix some integer $r \in \N_0$ and introduce the (column) vector
  $\mathbf{v}(x) = (v_1(x), \cdots, v_r(x))' \in \R^r$. Also let
  $g: \R \to \R$ be any function such that $v_k g \in L^1(p)$ for all
  $k=1, \ldots, r$.Then
  \begin{align}
    & \mathbb{E} \left[ \mathbf{v}(X) g(X) \Phi_p^{\ell}(u,X, v)  \right]
    \mathbb{E} \left[ \mathbf{v}'(X) g(X) \Phi_p^{\ell}(u,X, v)  \right]
      \nonumber \\
&    \qquad \qquad   =     \mathbb{E} \left[ \mathbf{v}(X)\mathbf{v}'(X) 
      \Phi_p^{\ell}(u,X, v)  \right]  \mathbb{E} \left[ g^2(X) \Phi_p^{\ell}(u, X,
                     v) \right]- R^\ell(u, v; \mathbf{v}, g), \label{eq:75}                    
  \end{align}
  where
  {{$ R^\ell(u, v; \mathbf{v}, g)  $ is the $r \times r$ matrix given by}}  
     \begin{equation}
    \label{eq:73}
    R^\ell(u, v; \mathbf{v}, g)  
     = \mathbb{E} \left[ (\mathbf{v}_3 g_4 - \mathbf{v}_4 g_3)  (\mathbf{v}_3 g_4 - \mathbf{v}_4 g_3)' 
      \Phi_p^{\ell}(u, X_3, X_4, v)
    \right]
  \end{equation}
  with
\begin{align}
  \label{eq:51}
  \Phi^\ell_p(u, x_1, x_2, v)  
 = \frac{\chi^{\ell}(u, x_1)\chi^{\ell^2}(x_1, x_2)\chi^{-\ell}(x_2, v)}{p(x_1)p(x_2)}. 
  \end{align}
  Here  {{$X_3, X_4$ denote two independent copies of $X$ and $\mathbf{v}_j = \mathbf{v}(X_j)$ so that $ v_{ij} = v_i(X_j)$,  and $g_j = g(X_j)$, $i=3, 4$.}} 
  When the context is clear, we abbreviate
  $R^\ell(u, v; \mathbf{v}, g) = R(u, v)$.
\end{lma}

\subsection{Papathanasiou-type expansion}
\label{sec:main-result}

Now the necessary ingredients are available to give the main result of this paper. We use the notation that for a vector $\mathbf{v}= (v_1, \ldots, v_r)' $ of functions, the operator $\Delta^\ell$ operates on each component, so that $\Delta^\ell \mathbf{v} =  (\Delta^\ell v_1, \ldots, \Delta^\ell v_r)' $. 

\begin{thm}\label{thm:var-}
  Fix $\ell \in \left\{ -1, 0, 1 \right\}$ and let
  $\pmb{\ell}=(\ell_n)_{n\ge1}$ be a sequence such that $\ell_n=0$ for all $n$ if $\ell =0$,
  otherwise $\ell_n\in \left\{ -1, 1 \right\}$ arbitrarily
  chosen. Let $(h_n)_{n\ge1}$ be a sequence of real valued functions
  $h_i: \R \to \R$ such that
  $\mathbb{P}[{\Delta^{-\ell_i} h_{i}}(X)>0] = 1$ for all $i\ge1$.
  Starting with some function $\mathbf{g}: \R \to \R^r$, we  recursively define the
  sequence $(\mathbf{g}_k)_{k\ge 0}$ by $\mathbf{g}_0(x)=\mathbf{g}(x)$ and
  $\mathbf{g}_{i}(x)={\Delta^{-\ell_i}\mathbf{g}_{i-1}(x)}/{\Delta^{-\ell_i}h_{i}(x)}$ for all
  $x \in \mathcal{S}(p)$. 
For any sequence $(x_j)_{j\ge1}$ we let {$\Phi_{0}^{\pmb{\ell}} (x_1, x_2)=1$ and}
    \begin{align}
   \lefteqn{ \Phi_{n}^{\pmb{\ell}}(x_1, x_3, \ldots, x_{2n-1}, x_{2n+1},
   x_{2n+2}, x_{2n}, \ldots, x_2) }\nonumber \\
&  \qquad = \frac{1}{\prod_{i=3}^{2n+2} p(x_i)} \chi^{\ell^2} (x_{2n+1}, x_{2n+2}) 
\prod_{i=1}^{n} \chi^{\ell_i} (x_{2i-1}, x_{2i+1}) \chi^{-\ell_i} (x_{2i +2}, x_{2i}).  \label{eq:phiex} 
  \end{align}
 Then, for all vectors of functions $\mathbf{f}: \mathbb{R}\rightarrow\mathbb{R}^r$ such that the expectations below exist, and all
$n \ge 1$, we have
\begin{align}\label{eq:varianceexpansion}
  &  \mathrm{Cov} \left[
\mathbf{f}(X)  \right]
          = 
 \sum_{k=1}^n (-1)^{k-1} \E\left[ 
                      \Delta^{-\ell_k} \mathbf{f}_{k-1}(X)
                                    \Delta^{-\ell_k} \mathbf{f}_{k-1}'(X) 
 \frac{\Gamma_k^{\pmb
                    \ell}\mathbf{h}(X) }{\Delta^{-\ell_{k}}
                        h_k(X)} 
                               \right] 
   + (-1)^n R_n^{\pmb\ell}(\mathbf{h})
\end{align}
where the derivatives are taken component-wise, and the weight
sequences are
\begin{align}
  &  \Gamma_k^{\pmb \ell} \mathbf{h}(x)  = \E\Bigg[ 
 (h_k(X_{2k})-h_k(X_{2k-1}))  
  \Phi^{\ell_{k}}_p(x_{2k-1}, x, x_{2k})   \Phi_{k-1}^{\pmb{\ell}}(X_1, \ldots,  X_{2k-1},  X_{2k}, \ldots, X_2) \nonumber \\
  &\quad \quad \quad  \quad \quad \quad  \prod_{i=1}^{k-1} 
\Delta^{-\ell_i} h_i(X_{2i+1}, X_{2i+2})   
    \Bigg]\label{eq:2}
\end{align}
 and 
\begin{align}
  R_n^{\pmb \ell}(\mathbf{h})& =  \E \Bigg[
 \left(\mathbf{f}_n(X_{2n+2})- \mathbf{f}_n(X_{2n+1})\right)
                               \left(\mathbf{f}_n(X_{2n+2})-
                               \mathbf{f}_n(X_{2n+1})\right)'
                               \nonumber \\
  & \qquad  \qquad     
\Phi_{n}^{\pmb{\ell}}(X_1, \ldots X_{2n+1},  X_{2n+2}, \ldots,  X_2) \prod_{i=1}^n   
\Delta^{-\ell_{i}} h_i{(X_{2i+1}, X_{2i+2})}   
         \Bigg] \label{eq:89}
\end{align}
where $\Delta^{\ell}h_k(x, y) = \Delta^{\ell}h_k(x)
\Delta^{\ell}h_k(y)$ and an empty product is set to 1.
\end{thm}


\begin{rmk}\label{cor:infvarexp}
  If $R_n^{\pmb \ell}(\mathbf{h}) \to 0$ as $n\to \infty$ then, under
  the conditions of Theorem \ref{thm:var-},
  \begin{align}\label{eq:infvarianceexpansion}
  &  \mathrm{Cov} \left[
\mathbf{f}(X)  \right]
          = 
 \sum_{k=1}^\infty (-1)^{k-1} \E\left[ 
                      \Delta^{-\ell_k} \mathbf{f}_{k-1}(X)
                                    \Delta^{-\ell_k} \mathbf{f}_{k-1}'(X) 
 \frac{\Gamma_k^{\pmb
                    \ell}\mathbf{h}(X) }{\Delta^{-\ell_{k}}
                        h_k(X)} 
                               \right].
\end{align}
In particular when $\mathbf{f}$ is a $d$th-degree polynomial, then
{$R_{n}^{\pmb \ell} (\mathbf{h})$ vanishes for $n \ge d$} and
\eqref{eq:varianceexpansion} is an exact expansion of the variance in
\eqref{eq:varianceexpansion} with respect to the
$\Gamma_k^{\pmb \ell} \mathbf{h}(x)$ functions ($k=1,\ldots,d$).
\end{rmk}

{
\begin{rmk}\label{rmk:incre}
  A stronger sufficient condition on the functions $h_i$ is that they
  be strictly increasing throughout $\mathcal{S}(p)$, in which case
  the condition $\Delta^{-\ell_i}h_{i}>0$ is guaranteed.  Under this
  assumption, the matrix $R_n^{\pmb\ell}(\mathbf{h})$ defined in
  \eqref{eq:89} is non-negative definite so that, in particular,
  taking $h_i = h$ for all $i \ge 1$ and fixing $r=2$ we recover the
  expansion \eqref{eq:3} as stated in the Introduction.
  \end{rmk}
}

\begin{rmk}
  When $\ell\neq 0$ then the condition that
  $\mathbb{P}[\Delta^{-\ell_i}h_i(X)>0]=1$ is itself also too
  restrictive because, as will have been made clear in the proof (see
  the Appendix), the recurrence only implies that
  $\Delta^{-\ell_i}h_i(x)$ needs to be positive on some interval
  $[a+\mathbf{a}_i; b-\mathbf{b}_i] \subset [a, b]$ where
  $\textbf{a}_i$ and $\textbf{b}_i$ are positive integers (they will
  be properly defined in \eqref{eq:56}). In particular when $\ell\neq 0$ the
sequence necessarily stops if $\mathcal{S}(p)$ is bounded, since after
a certain number of iterations the indicator functions defining
$\Phi_{n, j}^{\pmb\ell}$ will be 0 everywhere. 
\end{rmk}

Suppose that the assumption of Remark \ref{rmk:incre} applies, so that
the remainder is non negative definite. Then, taking $n=1$ in
\eqref{eq:varianceexpansion} gives an upper bound, and taking $n=2$
gives a lower bound, on the covariance, and the following holds
(stated again in the case $r=2$, for the sake of clarity).
  \begin{cor}
    Let all the conditions in Theorem \ref{thm:var-} prevail for
    ${{n}}=2$. Then
    \begin{align*}
 \E\left[             \Delta^{-\ell_1} f(X)
                                    \Delta^{-\ell_1} g(X) 
 \frac{\Gamma_1^{\ell_1}{h_1}(X) }{\Delta^{-\ell_1}h_1(X)} 
      \right] -
             \E\left[  \Delta^{-\ell_2} \bigg(
      \frac{\Delta^{-\ell_1}f(X)}{\Delta^{-\ell_1}h(X)}\bigg)
\Delta^{-\ell_2} \bigg(  \frac{\Delta^{-\ell_1}f(X)}{\Delta^{-\ell_1}h(X)}\bigg) 
 \frac{\Gamma_2^{\ell_1, \ell_2}(h_1, h_2)(X) }{\Delta^{-\ell_{2}}
                        h_2(X)}
                               \right] 
      \\
      \le \mathrm{Cov}[ f(X), g(X)] \le 
      \E\left[ 
                      \Delta^{-\ell_1} f(X)
                                    \Delta^{-\ell_1} g(X) 
 \frac{\Gamma_1^{\ell_1}{h_1}(X) }{\Delta^{-\ell_{1}}
                        h_1(X)} 
      \right].
    \end{align*}
  \end{cor}
  \begin{rmk}
    When $f=g$, the upper bound for $n=1$ is a weighted Poincar\'e
    inequality of the same essence as the upper bound provided in
    \cite{K85} (as revisited in \cite{ERS19vb1}), whereas the lower
    bound obtained with $n=2$ is of a different flavour.
  \end{rmk}

Of course such identities and expansions are only useful if the
weights are of a manageable form. This is exactly the topic of the
next section.

\section{About  the  weights  in Theorem \ref{thm:var-} }
\label{sec:weights-expansions}


The crucial quantities in Theorem \ref{thm:var-} are the sequences of
weights $\Gamma_k^{\pmb \ell} \mathbf{h}$ defined in \eqref{eq:2}.
For $k=1$,  the expression are  straightforward to obtain (see
equations \eqref{eq:9} for the continuous case $\ell_1 = 0$ and
\eqref{eq:10} for the discrete case $\ell_1 \in \left\{ -1, 1
\right\}$). 
 For larger $k$ the situation is not so straightforward. 
Relevance  
of the higher order terms in the covariance expansions
\eqref{eq:varianceexpansion} then hinges on the tractability of these
weights, which itself depends on the choice of functions
$h_1, h_2, \ldots$. In this section we restrict {attention} to the
(natural) choice $h_k(x) = h(x)$ for all $k$.  Then, writing
$ \Gamma_{k}^{\pmb \ell}h(x)$ instead of
$ \Gamma_{k}^{\pmb \ell}(h, h, \ldots)(x)$ we can express the sequence
of weights as
$\Gamma_{k}^{\pmb \ell}h(x) =: \mathbb{E} \left[ \gamma_k^{\pmb \ell}
  h(X_1, x, X_2) \right]$ where, for all $k \ge 1$, we set

\begin{align} \nonumber 
 \gamma_k^{\pmb \ell} h(x_1, x, x_2) 
 & =   \mathbb{E} \Big[  (h(X_{2k})-h(X_{2k-1}))
{\Phi^{\ell_{k}}_p(X_{2k-1}, x, X_{2k})   \Phi_{k-1}^{\pmb{\ell}}(x_1, X_3 \ldots,  X_{2k-1},  X_{2k}, \ldots, x_2)} 
    \\
    & \qquad \qquad  
    \prod_{i=1}^{k-1}\Delta^{-\ell_i} h(X_{2i+1}, X_{2i+2})
                                            \Big]. \label{eq:gammak} 
  \end{align}
  We now study  \eqref{eq:gammak} and the resulting expressions for
  the weights  under different sets of assumptions.

    \subsection{General considerations}
  \label{sec:gener-cons-cont}

  When no specific assumptions are made on $p$ or  $h$, we find it
  easier to separate the continuous case (i.e.\ $\ell = 0$) from the
  discrete one (i.e.\ $\ell \in\{-1,1\}$). 

  \subsubsection{The continuous case} 
  The continuous case is quite easy {as \eqref{eq:phiex} simplifies
    when all the test functions $h_i$ are equal and} the expressions
  follow directly from the structure of the weight sequence, which
  turn out to be straightforward iterated integrals. We note that such
  iterated integrals have a structure which may be of independent
  interest; all details are provided in the Appendix.

 \begin{lma}\label{lma:gamma2cont}
   Fix $\pmb\ell = (0, 0, \ldots)$ and let $h$ be non-decreasing. Then for all $k \ge 1$, 
   \begin{align}\label{eq:57cont}
     \gamma_k^0 h(x_1, x, x_2) & = (h(x)-h(x_1))^{k-1}(h(x_2)-h(x))^{k-1}(h(x_2)-h(x_1))
\frac{\mathbb{I}{[x_1\le x \le x_2]}}{p(x)k!(k-1)!}
   \end{align}
and 
\begin{equation}
  \label{eq:7}
  \Gamma_k^0h(x) = \frac{1}{k!(k-1)!} \frac{1}{p(x)}\mathbb{E}\left[ \big(h(x) - h(X_1)^{k-1} (h(X_2) -
  h(x)\big)^{k-1}\big(h(X_2) - h(X_1)\big) \mathbb{I}[X_1 \le x \le
  X_2] \right]. 
\end{equation}
   %
\end{lma}

Specific instantiations for different explicit distributions are given
in Section \ref{sec:illustrations-1}. We nevertheless note that,
letting $\nu(h)$ denote the mean $\mathbb{E}[h(X)]$ we get
\begin{align}\label{eq:9}
  \Gamma_1^0h(x) = \frac{1}{p(x)} \mathbb{E} \left[ (h(X_2) - h(X_1))
  \mathbb{I}[X_1\le x \le X_2] \right] = \frac{1}{p(x)} \mathbb{E}[
 {(\nu(h)-h(X))}
   \mathbb{I}[x \le X]] 
\end{align}
which one may recognize as the inverse of the canonical Stein operator
(see \eqref{eq:stcaninv}); in particular taking
$h(x) = \mathrm{Id}(x) = x$ the identity function, \eqref{eq:9} yields
the Stein kernel. For more information on the connection with Stein's
operators, see Section \ref{sec:conn-with-stein}.

\subsubsection{The discrete case} 
In the discrete case, simplifications {of
  $ \Gamma_{k}^{\pmb \ell}h(x)$} are more difficult {as
  \eqref{eq:phiex} depends strongly on the chosen sequence
  $\pmb\ell$}. Let
$\pmb \ell = (\ell_1, \ell_2, \ldots) \in \left\{ -1, +1
\right\}^{\infty}$. Recall the notations in \eqref{eq:6} and set
$a_{\ell_i} = a_i$, $b_{\ell_i} = b_i$ for $i\ge 1$.  Applying the definitions leads to 
  \begin{align}
&  \gamma_1^{\ell_1}h(x_1, x, x_2)  = (h(x_2)-h(x_1))
 \frac{\mathbb{I}[x_1 + a_1 \le x \le x_2 - b_1]}{p(x)}\label{eq:12}\\
& \gamma_2^{\ell_1, \ell_2}h(x_1, x, x_2) 
= \sum_{x_3 = x_1+ a_1}^{x-a_2}
\sum_{x_4 = x+ b_2}^{x_2-b_1}
(h(x_4)-h(x_3)) \Delta^{-\ell_1}h(x_3, x_4) 
\frac{\mathbb{I}[x_1 + a_1+a_2 \le x \le x_2 - b_1-b_2]}{p(x)}.\label{eq:13}
  \end{align}
  In order to generalize to arbitrary $k \ge3$,   we introduce 
  \begin{align}
  \label{eq:56}
  \mathbf{a}_k=  \sum_{i=1}^k a_i  \mbox{ and }  
   \mathbf{b}_{k}  = \sum_{i=1}^k b_i.
  \end{align}
  Note that $\mathbf{a}_{k} ( = \mathbf{a}_k(\pmb \ell)) $ counts the number of ``$+$'' in the first
  $k$ components of $\pmb \ell$ and $\mathbf{b}_{k}  ( = \mathbf{b}_k(\pmb \ell))$ counts the
  corresponding number of ``$-$'', {so that}
  $\mathbf{a}_k + \mathbf{b}_k = k$. 
Then  for $k \ge 2$ we have (sums over empty sets are set to 1): 

   \begin{align*}
    \gamma_{k}^{\pmb \ell}h(x_1, x, x_2)     & =
\left( \sum_{x_3=x_1+\mathbf{a}_{k-1}}^{x-a_{k}}\sum_{x_4=x+b_{k}}^{x_2 -\mathbf{b}_{k-1}}  (h(x_4)-h(x_3))
                                            \Delta^{-\ell_{k-1}}h(x_3,
                                                 x_4)  
          \sum_{x_{5}=x_1+\mathbf{a}_{k-2}}^{x_3-a_{k-1}}\sum_{x_6=x_4+b_{k-1}}^{x_2
                                                 -
                                                 \mathbf{b}_{k-2}}\Delta^{-\ell_{k-2}}h(x_5,
                                                 x_6)  \right.\\
  &    \qquad \cdots \left. \sum_{x_{2k-1}=x_1+a_{1}}^{x_{2k-3}-a_{2}}\sum_{x_{2k+1}=x_{2k-2}+b_{2}}^{x_2
                                            - b_1} \Delta^{-\ell_1}h(x_{2k-1},
                                        x_{2k})  \right)  \frac{ \mathbb{I}[x_1 +
                                            \mathbf{a}_{k}  \le x
                                         \le x_2 -  \mathbf{b}_{k} ]    }{p(x)}
  \end{align*}
  for all $x \in \mathcal{S}(p)$ and all $x_1, x_2$. This is a proof of the
  next result.
  \begin{prop}\label{lma:gamma2}
    Instate all previous notations.  For all $k \ge 1$,
  \begin{align*}
     \gamma_{k}^{\pmb \ell}h(x_1, x, x_2)     & =
\left( \sum_{x_3=x_1+\mathbf{a}_{k-1}}^{x-a_{k}}
 \sum_{x_4=x+b_{k}}^{x_2 -\mathbf{b}_{k-1}}
 (h(x_4)-h(x_3)) \psi^{\pmb \ell}_{k-1}h(x_1, x_3, x_4, x_2) \right) \frac{ \mathbb{I}[x_1 + \mathbf{a}_{k}  \le x  \le x_2 -  \mathbf{b}_{k} ]    }{p(x)}
  \end{align*}
  where $\psi^{\pmb \ell}_0h(x_1, x_3, x_4, x_2) = 1$ and, for
  $k \ge 2$,
  $\psi^{\pmb \ell}_{k-1}h(x_1, x_3, x_4, x_2) = \psi_{{k-1},1}^{\pmb
    \ell}h(x_1, x_3) \psi_{{k-1},2}^{\pmb \ell}h(x_4, x_2) $ and
  \begin{align*}
& \psi_{{k-1},1}^{\pmb \ell}h(x_1, x_3) =     \Delta^{-\ell_{k-1}}h(x_3)
                   \sum_{x_{5}=x_1+\mathbf{a}_{k-2}}^{x_3-a_{k-1}}
                   \left(  \Delta^{-\ell_{k-2}}h(x_5)
                   \sum_{x_{7}=x_1+\mathbf{a}_{k-4}}^{x_5-a_{k-2}} \left(  \cdots
\sum_{x_{2k-1}=x_1+a_{1}}^{x_{2k-3}-a_{2}}  
                   \Delta^{-\ell_1}h(x_{2k-1}) \right) \right)\\
    & \psi_{{k-1},2}^{\pmb \ell}h(x_4, x_2) =   \Delta^{-\ell_{k-1}}h(x_4)
      \sum_{x_6=x_4+b_{k-1}}^{x_2- \mathbf{b}_{k-2}} \left( 
      \Delta^{-\ell_{k-2}}h(x_6) \sum_{x_8=x_6+b_{k-2}}^{x_2-
      \mathbf{b}_{k-3}} \left(  \cdots
\sum_{x_{2k}=x_{2k-2}+b_{2}}^{x_2 - b_1}
  \Delta^{-\ell_1}h(x_{2k})  \right)\right)
  \end{align*}
  for all $x_1 + \mathbf{a}_{k-1} \le x_3 \le x_4 \le x_2 -\mathbf{b}_{k-1}$. 
\end{prop}

 Taking expectations
in \eqref{eq:12} and \eqref{eq:13} we obtain
\begin{align}
  \label{eq:10}
&   \Gamma_1^{\ell_1}h(x) = \frac{1}{p(x)} \mathbb{E} \left[ (h(X_2) -
  h(X_1)) \mathbb{I}[X_1 + a_1 \le x \le X_2 - b_1] \right]\\
  & \Gamma_2^{\ell_1, \ell_2}h(x)  = \frac{1}{p(x)} \mathbb{E} \left[ \sum_{x_3 = X_1+
                                             a_1}^{x-a_2}\sum_{x_4 =
                                             x+ b_2}^{X_2-b_1}
                                             (h(x_4)-h(x_3))
                                             \Delta^{-\ell_1}h(x_3, x_4) 
{\mathbb{I}[X_1 + \mathbf{a}_2 \le x
                   \le X_2 - \mathbf{b}_2]} \right].
\end{align}
The expressions for higher orders are easy to infer, but this seems to
be the best we can do because the expressions in
Proposition~\ref{lma:gamma2} are obscure and, unfortunately, we have
not been able to devise a formula as transparent as \eqref{eq:57cont}
for general $h$ in the discrete case. Nevertheless, simple manageable
expressions are obtainable for certain specific choices of $h$,
particularly the case $h(x) = \mathrm{Id}(x)$ as we shall see in
Section
\ref{sec:some-gener-cons}.  

\subsubsection{Connection with Stein operators}
\label{sec:conn-with-stein}

In   \cite{ERS19vb1} we introduced the \emph{canonical inverse Stein
  operator}
  \begin{equation}\label{eq:stcaninv}
    \mathcal{L}_p^{\ell}h(x) =  \mathbb{E}\bigg[(h(X_1) -
    h(X_2)) \Phi_p^{\ell}(X_1, x, X_2)\bigg]
  \end{equation}
  for $h \in L^1(p)$ and $X_1, X_2$ independent copies of $X \sim p$.
  This operator has the property of yielding solutions to so-called
  Stein equations, both in discrete and continuous setting; it has
  many important properties within the context of Stein's method. In
  particular it provides generalized covariance identities and, when
  $h(x) = \mathrm{Id}(x)$ is the identity function, it provides
  \begin{align}
    \label{eq:14}
    \tau^{\ell}_p(x) =  {-   \mathcal{L}_p^{\ell} \mathrm{Id}(x)}
  \end{align}
  the all-important Stein kernel of $p$. This function, first
  introduced in \cite{stein1986}, has long been known to provide a
  crucial handle on the properties of $p$ and is now studied as an
  object of intrinsic interest, see e.g.\ \cite{courtade2017existence,fathi2018stein}.

  From \eqref{eq:9} and \eqref{eq:10}, we immediately recognize that
  $ \Gamma_1^{\ell_1}h(x) ={-\mathcal{L}_p^{\ell} h(x)}$, in other words
  the first order weight in our expansion is given by a Stein
  operator. There is also a connection between $\Gamma_k^{\pmb \ell}h$
  and ``higher order'' Stein kernels. To see this, restrict to the
  continuous case $\pmb \ell = 0$ and introduce
 $H^{k}_x(y) = (h(y) - h(x))^{k}/k!$. Then 
  \eqref{eq:57cont} becomes 
\begin{align}
\Gamma_{k}^{\pmb 0}h(x) 
& = (-1)^k \left( 
\mathbb{E}\big[H^{k-1}_x(X)\big] \mathcal{L}_p^0H^{k}_x(x) 
- \mathbb{E}\big[H^{k}_x(X)\big] \mathcal{L}_p^0H^{k-1}_x(x) 
    \right) 
    \label{eq:57cont2} 
\end{align}
(see the Appendix for a proof). In the case $h(x) = x$ the expression
\eqref{eq:57cont2} 
simplifies to Papathanasiou's weights
from \eqref{eq:101}.  This allows to make the connection between
considerations related to Stein's method and the weights appearing in
the expansions, 
{as} has already been observed 
(see e.g.\ \cite{APP07}). We do not pursue this line of research here,
{{except to point out that}}
our result provides a framework to the important
works  
\cite{papathanasiou1988variance,korwar1991characterizations,johnson1993note,APP07,AfBalPa2014},
which focus on particular families of distributions, see
Sections~\ref{sec:pearson} and \ref{sec:ord}.  Further study of this connection, in line
e.g.\ with \cite{fathi2018higher}, is outside the scope of this paper
and deferred to a future publication.

\subsection{Handpicking the test functions}
\label{sec:some-gener-cons}
We now focus on  particular choices of $h$. 
To begin with, we consider  the most intuitive choice (and the only one studied in
the literature): $h(x) = \mathrm{Id}(x)$. {{In this case we
    abbreviate $\Gamma_k^{\mathbf{\ell}} \mathbf{h} (x) =
    \Gamma_k^{\mathbf{\ell}}  (x)$.}} If $\pmb \ell = \pmb 0$ we have
  \begin{equation*}
    \Gamma_k^{\pmb 0}(x) = \frac{1}{ k!(k-1)! p(x)}\mathbb{E} \left[
      (X_2-x)^{k-1}(x-X_1)^{k-1}(X_2-X_1) \mathbb{I}[X_1 \le x \le X_2] \right].
  \end{equation*}
  The discrete case is less transparent, but direct computations for 
the first two weights in the discrete case lead to 
\begin{align*}
  & \Gamma_1^{\ell_1}(x) = \frac{1}{p(x)} \mathbb{E}[(X_2-X_1)
  \mathbb{I}[X_1+a_1 \le x \le X_2-b_1] ]\\
  &\Gamma_2^{\ell_1, \ell_2}(x) = \frac{1}{2p(x)} \mathbb{E}[ (X_2 - x
    -\mathbf{b}_2+1) (x-X_1 -
    \mathbf{a}_{2}+1)
    (X_2-X_1)
  \mathbb{I}[X_1+ \mathbf{a}_2 \le x \le X_2-\mathbf{b}_2] ].
\end{align*}

More
 generally we have the following. 
\begin{lma}\label{lma:identitycontdisc}
  If $\pmb \ell \in \left\{ -1, 1 \right\}^{\infty}$ then for all
  $k \ge 1$ 
   \begin{align}
     & \Gamma_{k}^{\pmb \ell}  (x) = 
\frac{1}{p(x) k! (k-1)!} \mathbb{E} \left[ (X_{2}-x-\mathbf{b}_{k}+1)^{[k-1]}
       (x-X_{1}-\mathbf{a}_k+1)^{[k-1]}(X_{{2}}-X_{1}) 
{ \mathbb{I}[X_1 + \mathbf{a}_{k}  \le x \le X_2 -
\mathbf{b}_{k}]} \right].  \label{eq:57Discretea} 
   \end{align}
We can unify the continuous and the discrete settings, to reap
\begin{align*}
\Gamma_k^{\pmb \ell}(x) =   \mathbb{E}
    \left[(X_{2}-x)_{\{k-1;\pmb\ell\}}
    (x-X_{1})^{\{k-1;\pmb\ell\}}(X_{{2}}-X_{1}) 
    \frac{ \mathbb{I}[X_1 + \mathbf{a}_{k}  \le x \le X_2 -
  \mathbf{b}_{k} ]}{p(x) k!(k-1)!}  \right]
\end{align*}
   where $f_{\{k,\pmb\ell\}}(x) =
   \prod_{j=1}^{k}f(x+\mathbf{a}_k-|\ell|j)$ and
   $f^{\{k,\pmb\ell\}}(x)  = \prod_{j=1}^{k}f(x-\mathbf{a}_k+|\ell|j)$ or equivalently
   \begin{align*}
   f_{\{k,\pmb\ell\}}(x) &= \begin{cases}
   f(x)^k & \text{ if } {\pmb\ell}={\pmb 0}, \\
   \prod_{j=1}^{k}f(x+\mathbf{a}_k-j) = f_{[k]}(x+\mathbf{a}_k-1) &\text{ else};
   \end{cases}  \\
f^{\{k,\pmb\ell\}}(x) &= \begin{cases}
   f(x)^k & \text{ if } {\pmb\ell}={\pmb 0}, \\
   \prod_{j=1}^{k}f(x-\mathbf{a}_k+j) = f^{[k]}(x-\mathbf{a}_k+1) &\text{ else}.
   \end{cases}    
   \end{align*}
    {and the empty product equals 1.}
  \end{lma}

  \begin{rmk}
As already noted in Section \ref{sec:conn-with-stein}, the expression
of the weights in the continuous case is already known and can be
traced back to works as early as  \cite{papathanasiou1988variance};
the expression for the discrete case (namely equation \eqref{eq:57Discretea}) is new,
although a version with $\pmb \ell = (-1, -1, -1, \ldots)$ is
available from \cite{APP07}.     
  \end{rmk}

Another natural choice in the continuous case $\pmb \ell = 0$, of
increasing function $h$ to plug into the weights is $h(x) = P(x)$ with
$P$ the cdf of $p$. Then the following holds.
\begin{lma} \label{lma:cdf}
If $\pmb \ell = 0$ and $X \sim p$ has cdf $P$ then 
$    \Gamma_{k}^0P(x) = \frac{1}{k!(k+1)!p(x)}P(x)^k(1-P(x))^k.$
\end{lma}

A final natural choice occurs whenever $p$ is
  log-concave. Indeed in this case the function $h_1=-(\log
  p)' 
  $ is
  increasing.  
  In particular,
  $\Gamma_1^0h_1(x)= -\mathcal{L}_p^0 h_1 (x) = 1$, which allows us to
  rewrite the first order expansion as
\begin{align*}
\mathrm{Cov} \left[f(X),g(X)\right] 
& = \E\left[  \frac{f'(X)g'(X)}{-(\log p)''(X)} \right]  - R_1^{0}(\mathbf{h}).
\end{align*}
This expression generalizes the Brascamp-Lieb inequality from
\cite{ERS19vb1}, see also \cite{ERS19vb1} for more information. {For
  simple expressions of $R_1^{0}(\mathbf{h})$ one may like to choose
  $h_2= h_3 = \cdots = \mathrm{Id}$. This example thus benefits from
  the flexibility in choosing a sequence of functions $\mathbf{h}$.}


\subsection{Illustrations}
\label{sec:illustrations-1}

\subsubsection{The weights for  Integrated Pearson  family}
\label{sec:pearson}

  \begin{defn}[Integrated Pearson] \label{defn:lagr-ident-infin-1}
We say that $X \sim p$ belongs to the integrated Pearson family  if $X$ is absolutely continuous and 
there exist $\delta, \beta, \gamma \in \R$ not all equal to 0 such
that  
  $\tau_p^{{\mathbf{0}}}(x)\big(:=-\mathcal{L}_p^{0}\mathrm{Id}(x)\big)=\delta x^2 + \beta x
  + \gamma$ for all $x \in \mathcal{S}(p)$. 
\end{defn}
Definition \ref{defn:lagr-ident-infin-1} corresponds to the continuous
Pearson systems, a.k.a.\ integrated Pearson, as studied e.g.\ in
\cite{ap14} (see their Definition 1.1).
  The following results hold (to facilitate comparison of the results
  we use  the same notations as in \cite{ap14}).

\begin{prop}\label{prop:perason}
  If $X \sim p$ is integrated Pearson distributed with Stein kernel
  $\tau_{p} (x)= {{\tau_p^{\mathbf{0}}(x) }} =-\mathcal{L}_p^0(\mathrm{Id})= \delta x^2 + \beta x + \gamma$ then
\begin{equation}
\Gamma_{k}^0(x) = \frac{\tau_p(x)^k}{k! \prod_{j=0}^{k-1}(1-j \delta)}.
\end{equation}
\end{prop}

{
The coefficient $(\delta, \beta,\gamma)$ of the Stein kernel are explicitly given in \cite[Table 3]{ERS19vb1}. These coefficients allow us to directly obtain the infinite expansion of covariance for the integrated Pearson family. We give the expansions for two distributions in the following examples. 
}

\begin{exm}[Normal expansion] \label{ex:normbounds} {The standard
    normal distribution ${\phi}$ is an element of the integrated Pearson family
    with $\delta=0, \beta =0, $ and $\gamma=1$.}  Direct computations
  show that if $X \sim \mathcal{N}(0, 1)$ then $\tau_{\phi}(x) = 1$ so
  that $\Gamma_k^0(x) = \frac{1}{k!}$ for all $k$ and
  \begin{align*}
    \mathrm{Cov}[f(X), g(X)] &  =
                               \sum_{k=1}^{\infty}\frac{(-1)^{k-1}}{k!}\mathbb{E}
                               \left[ f^{(k)}(X)g^{(k)}(X)\right],
  \end{align*}
  which extends {the variance expansion} \eqref{eq:104} {to a covariance expansion.} 
  
\end{exm}

\begin{exm}[Beta expansion] \label{ex:betabounds}
{The Beta$(a, b)$ distribution is an element of the integrated Pearson family with $\delta=- \frac{1}{a +b}, \beta =\frac{1}{a+b}, $ and $\gamma=0$; then  $\tau_{\rm{Beta}(a,b)}(x) = \frac{x(1-x)}{a+b}$.}
  Direct computations show that if $X \sim \mathrm{Beta}(a,b)$ then
    $\Gamma_k^0(x) =  (x(1-x))^k/(k!(a+b)^{[k]})$ for $k\geq 1$, so that 
 \begin{align*}
   \mathrm{Cov}[f(X), g(X)] &  =
                              \sum_{k=1}^{\infty}\frac{(-1)^{k-1}}{k!(a+b)^{[k]}}\mathbb{E} \left[ f^{(k)}(X)g^{(k)}(X)X^k(1-X)^k\right].
 \end{align*}
    
\end{exm}

\subsubsection{The weights for Cumulative Ord family}\label{sec:ord}

{In this section the superscript $+$ denotes $\ell=1$ and the
  superscript $-$ denotes $\ell = -1$.}

\begin{defn}[Cumulative Ord families] \label{defn:lagr-ident-infin-1ORD}
  We say that $X \sim p$ belongs to the cumulative Ord family if $X$
  is discrete and there exist $\delta, \beta, \gamma \in \R$ not all
  equal to 0 such that
  $\tau_p^{{-}}(x)\big(:=-\mathcal{L}_p^{{-}}(\mathrm{Id})\big)=\delta
  x^2 +\beta x + \gamma$ for all $x \in \mathcal{S}(p)$. {(It follows  that for this distribution $p$, 
$\tau^{{+}}_p(x)=\frac{p(x-1)}{p(x)}\tau^-_p(x-1) = x(\delta x + \beta + 1)$.)
}
\end{defn}

%
The following results hold (to facilitate comparison of the results we
use the exact same notations as in \cite{APP07}).

\begin{prop}\label{prop:perasondisc}
 If $X \sim p$ is cumulative Ord
distributed with 
$\tau_p^{{-}}(x)
=\delta x^2 + \beta x + \gamma$ {(and hence $\tau^{{+}}_p(x)= x(\delta x + \beta + 1)$)}, then
\begin{equation}
\Gamma_{k}^{\pmb \ell}(x) = \frac{1}{k! \prod_{j=0}^{k-1}(1-j \delta)} \left(\tau_p^{{+}}(x)\right)_{[\mathbf{a}_k]}
\left(\tau_p^{{-}}(x)\right)^{[\mathbf{b}_k]}.
\end{equation}
\end{prop}


\begin{rmk}
By taking only $k$ forward difference, i.e., ${\pmb
  \ell}=(-1,\ldots,-1)$, we deduce the result of \cite[Theorem
4.1]{APP07}. In particular, their Table 1 illustrates the expression
of $\Gamma^{\pmb \ell}_k(x)$ for some discrete distributions from
the cumulative Ord family. Tables at the end of
\cite{ERS19vb1}  give explicit expressions of  Stein kernels for many standard distributions.  
\end{rmk}

In the discrete case, there is much more flexibility in the
construction of the bounds as any permutation of $+1$ and $-1$ is
allowed 
for every $k$, leading to: 
\begin{align*}
    \mathrm{Var}[g(X)] &  =  \mathbb{E} \left[
                         \Gamma_1^+(X)    (\Delta^{-}g(X))^2 \right] -
                         R_1^{+}  = \mathbb{E} \left[
                         \Gamma_1^-(X)    (\Delta^{+}g(X))^2 \right] -
                         R_1^-
  \end{align*}
  and for an order 2 expansion, 
{for any of the four choices of $(\ell_1, \ell_2) \in \{ -1, +1\}^2$, }{\begin{align*}
 \mathrm{Var}[g(X)] &  =  \mathbb{E} \left[ \Gamma_1^{\ell_1} (X)(\Delta^{- \ell_1}g(X))^2 \right] -
 \mathbb{E} \left[ \Gamma_2^{\ell_1, \ell_2}(X) (\Delta^{-\ell_1, -\ell_2}g(X))^2 \right]  + 
                         R_2^{\ell_1, \ell_2}
\end{align*}
}{where we use the concise notation $\Delta^{\ell_1, \ell_2}g(X)$ for $\Delta^{\ell_2} \left( \Delta^{\ell_1}g(X)\right)$.}

\begin{exm}[Binomial expansion]\label{ex:binomgamma}
The Binomial$(n,\theta)$ distribution is an element of the cumulated
  Ord family with $\delta=0, \beta =-\theta, $ and $\gamma=n\theta$;
  its Stein kernels are $\tau^{-}(x) = \theta(n-x)$ and $\tau^{+}(x) =
  (1-\theta)x$. Hence 
  \begin{equation*}
    \Gamma_1^+(x) = (1-\theta) x, \quad   \Gamma_1^-(x) =  \theta(n-x)
  \end{equation*}
    so that the order 1 expansions are 
  \begin{align}
    \mathrm{Var}[g(X)] &  =  (1-\theta) \mathbb{E} \left[X (\Delta^{-}g(X))^2 \right] - R_1^{+} \label{var1}\\
    & = \theta \mathbb{E} \left[(n-X)  (\Delta^{+}g(X))^2 \right] -  R_1^-; \label{var2} 
  \end{align}
  choosing a linear combination of
      \eqref{var1} and \eqref{var2} with weights $\theta$ and
      $1-\theta$, respectively, yields  
   \begin{align}\label{eq:5}
    \mathrm{Var}[g(X)] &  =  n \theta (1 - \theta) \mathbb{E} \left[ \frac{X}{n}
        (\Delta^-g(X) )^2+ \frac{n-X}{n} (\Delta^+g(X))^2 \right] -
                         \theta R_1^+ - (1-\theta) R_1^-.
   \end{align}
    We note that \cite[Theorem 1.3]{hillion2011natural} introduce the
 ``natural binomial derivative''
$   \nabla_ng(x) =  \frac{x}{n}  \Delta^-g(x) + \frac{n-x}{n} \Delta^+g(x)$
and  prove -- by arguments which are specific to the binomial
distribution --  the Poincar\'e inequality
$$\mathrm{Var}[g(X)]   \le  n \theta (1 - \theta) \mathbb{E} \left[
  \big(\nabla_ng(X)\big)^2\right].$$ The connection with  \eqref{eq:5}
is easy to see because    (see e.g.\ \cite[Remark
  3.3]{hillion2011natural})
 \begin{align*}
                            \big(\nabla_ng(x)\big)^2  =
\frac{x}{n}
        (\Delta^-g(x) )^2+ \frac{n-x}{n} (\Delta^+g(x))^2  -
   \frac{x(n-x)}{n^2} (\Delta^{+-}g(x))^2.
 \end{align*}
Moving to the second order, direct computations show that
    \begin{align*}
&       \Gamma_2^{+,+}(x) = \frac{1}{2}(1-\theta)^2x(x-1) \mathbb{I}[1
                     \le x \le n],  \quad 
 \Gamma_2^{+,-}(x) =  \Gamma_2^{-,+}(x) =\frac{1}{2}\theta
                                      (1-\theta) x(n-x) \mathbb{I}[0
                                      \le x \le n]\\
      & \mbox{and }\Gamma_2^{-,-}(x) = \frac{1}{2}\theta^2(n-x)(n-x-1)
        \mathbb{I}[0 \le x \le n-1]                                      
    \end{align*}
    leading to the order 2  expansions
\begin{align*}
\mathrm{Var}[g(X)] 
&  =  (1-\theta) \mathbb{E} \left[ X  (\Delta^{-}g(X))^2 \right] - 
\frac{1}{2}(1-\theta)^2 \mathbb{E} \left[ X(X-1)  (\Delta^{- -}g(X))^2 \right]  +  R_2^{++}\\
&  =  (1-\theta) \mathbb{E} \left[ X  (\Delta^{-}g(X))^2 \right] -
\frac{1}{2}\theta (1-\theta)  \mathbb{E} \left[X(n-X) (\Delta^{-+}g(X))^2 \right]  +  R_2^{+-}\\
&  = \theta  \mathbb{E} \left[(n-X)   (\Delta^{+}g(X))^2 \right] -
\frac{1}{2}\theta (1-\theta) \mathbb{E} \left[ X(n-X) (\Delta^{+-}g(X))^2 \right]  + R_2^{-+} \\
&  = \theta\mathbb{E} \left[ (n-X)   (\Delta^{+}g(X))^2 \right] -
 \frac{1}{2}\theta^2 \mathbb{E} \left[ (n-X-1)(n-X) (\Delta^{+ +}g(X))^2 \right]  + R_2^{--}.
\end{align*}
Using the notation $\nabla_n$ from above, we deduce from a combination
of the second and third identities the lower variance bound
\begin{align*}
  \mathrm{Var}[g(X)] \ge n \theta (1-\theta) \left\{  \mathbb{E} \left[
  \big(\nabla_ng(X)\big)^2 \right] -  
 \frac{n-2}{2} 
  \mathbb{E} \left[ \frac{X(n-X)}{n^2} (\Delta^{+-}g(X))^2  \right]\right\}  .
\end{align*}
Combining these inequalities yields that for $0 < \theta < 1$, 
\begin{align*}
\mathbb{E} \left[
  \big(\nabla_ng(X)\big)^2 \right] - \frac{n-2}{2} 
  \mathbb{E} \left[ \frac{X(n-X)}{n^2} (\Delta^{+-}g(X))^2  \right]  \le  \frac{ \mathrm{Var}[g(X)]}{n \theta (1-\theta)} \le \mathbb{E} \left[
  \big(\nabla_ng(X)\big)^2 \right]. 
\end{align*}
\end{exm}

\subsubsection{Examples which are not integrated Pearson or cumulative Ord distributions}
\label{sec:some-examples}

\begin{exm}[Laplace expansion] \label{ex:laplacebounds} Direct
  computations show that if $X \sim \mathrm{Laplace}(0, 1)$ (i.e.\
  $p(x) = e^{-|x|}/2$ on $\R$)  then 
  {$\Gamma_1^0(x) = 1+|x|$ and
    $\Gamma_2^0(x) = \frac{1}{2}x^2 + |x| + 1$} so that the first two
  bounds become
  \begin{align*}
    \mathrm{Var}[g(X)] &  =  
   { \mathbb{E} \left[(1+|X|)g'(X)^2 \right] - R_1 }\\
& =  \mathbb{E} \left[(1+|X|) g'(X)^2 \right] -
    \mathbb{E} \left[ (1+|X|+ X^2/2) g''(X)^2 \right] +  R_2.
  \end{align*}
    Despite this distribution not being a member of the Pearson
    family, the general expression for $\Gamma_k$ is quite {simple:}
  \begin{equation*}
    \Gamma_k^0(x) = \sum_{j=0}^k \frac{|x|^j}{j!}.
  \end{equation*}
  The structure of this sequence seems to indicate that this
  distribution is of a different nature {than integrated Pearson
    distributions}; this is also illustrated in the properties of the
  corresponding Stein operator (which is best described as a second
  order differential operator), see
  \cite{eichelsbacher2015malliavin,PR12}.
\end{exm}

\begin{exm}[Rayleigh expansion] \label{ex:rayleighbounds} Direct
  computations show that if $X \sim \mathrm{Rayleigh}(0, 1)$ (i.e.\
  $p(x) = x e^{-x^2{/2}}$ on $\R^+$) then $\tau_p^0(x)$ does not take on
  an agreeable form.
  Nevertheless the choice $h(x) = x^2$ leads to
  \begin{equation*}
    \frac{\Gamma_k^0h(x)}{h'(x)} =  \frac{2^{k-2}}{k!}x^{2(k-1)}.
  \end{equation*}

\end{exm}

\begin{exm}[Cauchy expansion]
  The {standard} Cauchy distribution  lacks
  moments; nevertheless taking $h(x) = \arctan(x)$ leads to
  \begin{equation*}
    \frac{\Gamma_k^0(x)}{h'(x)} = \frac{1}{4^k(k+1)!(k)!}(1+x^2)^2
    \left( \pi^2-4\arctan(x)^2 \right)^k. 
  \end{equation*}
  
\end{exm}

\begin{exm}[Levy expansion] {The $pdf$ of the standard Levy distribution is given by 
$ (2 \pi)^{-\frac12} e^{\frac{1}{2x}} x^{-\frac32} $. Similarly as in the previous example, } 
   taking $h(x) = P(x)$, 
  \begin{equation*}
    \frac{\Gamma_k^0(x)}{h'(x)} = \binom{k+1}{2}\frac{1}{k!(k+1)!} \pi e^{1/x} x^3
    \big((1- P(x)) P(x) \big)^k. 
  \end{equation*}
  
\end{exm}

\section*{Acknowledgements}

The research of YS was partially supported by the Fonds de la
Recherche Scientifique -- FNRS under Grant no F.4539.16.  ME
acknowledges partial funding via a Welcome Grant of the Universit\'e
de Li\`ege and via the Interuniversity Attraction Pole StUDyS (IAP/P7/06). 
YS also thanks Lihu Xu for organizing the ``Workshop on
Stein's method and related topics'' at University of Macau in December
2018, and where this contribution was first presented. GR and YS also
thank Emilie Clette for fruitful discussions on a preliminary version
of this work. YS thanks Jean-Pierre Schneiders for discussions on the
weights. We also thank Benjamin Arras for several pointers to
relevant literature, as well as corrections on the first draft of the
paper.

\bibliographystyle{spmpsci} 
\bibliography{biblio_ysersvc}

\appendix
\section{Proofs}

\begin{proof}[Proof of Lemma \ref{lma:var1}]
The equivalence between \eqref{eq:27bis} and \eqref{eq:27} follows from the fact that
  $\mathbb{I}[X_1<X_2]+\mathbb{I}[X_1=X_2]+\mathbb{I}[X_1>X_2]=1$
  and
  $$\E\left[ \big(f(X_2)-f(X_1) \big) \big(g(X_2)-g(X_1) \big)
    \mathbb{I}{[X_1<X_2]} \right]=\E\left[ \big(f(X_2)-f(X_1) \big)
    \big(g(X_2)-g(X_1) \big) \mathbb{I}{[X_2<X_1]} \right].$$
{Without loss of generality in \eqref{eq:27bis} it can be assumed that $\E [f(X)] = \E [g(X)] =0$. Evaluating the expectation \eqref{eq:27bis} through expanding the product yields the assertion.}
%
\end{proof}

\begin{proof}[Proof of Lemma \ref{lma:matrixcauchyschwarz}]
First, from {\eqref{eq:76} in} Lemma \ref{lem:chirules} it follows directly that 
\begin{align}
  \label{eq:51b}
  \Phi^\ell_p(u, x_1, x_2, v) \mathbb{I}[x_1 \ne x_2]
  =  \mathbb{I}[x_1 \ne x_2] \chi^{\ell^2}(x_1, x_2)  \Phi^\ell_p(u, x_1, v)  \Phi^\ell_p(u, x_2, v). 
  \end{align}
 With the abbreviations as introduced in the statement of the lemma, the $(i,j)$ entry of the $r \times r$ matrix  $R(u,v)$ is 
   \begin{eqnarray*}
{(R(u,v))_{i,j} }
    &:=&   \mathbb{E} \left[ (v_{i3} g_4 - v_{i4} g_3) (v_{j3} g_4 - v_{j4} g_3)  \Phi_p^{\ell}(u, X_3, X_4, v)
     \right] \\
     &=&  \mathbb{E} \left[\mathbb{I}[X_3 \ne X_4]  (v_{i3} g_4 - v_{i4} g_3) (v_{j3} g_4 - v_{j4} g_3)  \chi^{\ell^2}(X_3, X_4)  \Phi^\ell_p(u, X_3, v)  \Phi^\ell_p(u, X_4, v)
     \right] ,
   \end{eqnarray*} 
   where we used \eqref{eq:51b} in the last step. 
   Next, again using Lemma \ref{lem:chirules}, 
   $ \mathbb{I}[x_1 \ne x_2] ( \chi^{\ell^2}(x_1, x_2)  +  \chi^{\ell^2}(x_2, x_1))
 =   \mathbb{I}[x_1 \ne x_2]$
   and by symmetry, 
    \begin{eqnarray*} 
    \lefteqn{ \mathbb{E} \left[\mathbb{I}[X_3 \ne X_4]  (v_{i3} g_4 - v_{i4} g_3) (v_{j3} g_4 - v_{j4} g_3)  \chi^{\ell^2}(X_3, X_4)  \Phi^\ell_p(u, X_3, v)  \Phi^\ell_p(u, X_4, v)
     \right]}\\
     &=&  \mathbb{E} \left[\mathbb{I}[X_4 \ne X_3]  (v_{i3} g_4 - v_{i4} g_3) (v_{j3} g_4 - v_{j4} g_3)  \chi^{\ell^2}(X_4, X_3)  \Phi^\ell_p(u, X_3, v)  \Phi^\ell_p(u, X_4, v)
     \right] .
     \end{eqnarray*} 
     Thus 
       \begin{eqnarray*}
{2 (R(u,v))_{i,j} }
    &=&  \mathbb{E} \left[\mathbb{I}[X_3 \ne X_4]  (v_{i3} g_4 - v_{i4} g_3) (v_{j3} g_4 - v_{j4} f_3)  \chi^{\ell^2}(X_3, X_4)  \Phi^\ell_p(u, X_3, v)  \Phi^\ell_p(u, X_4, v)
     \right] \\
     &&+  \mathbb{E} \left[\mathbb{I}[X_4 \ne X_3]  (v_{i3} g_4 - v_{i4} g_3) (v_{j3} g_4 - v_{j4} g_3)  \chi^{\ell^2}(X_4, X_3)  \Phi^\ell_p(u, X_3, v)  \Phi^\ell_p(u, X_4, v)
     \right] \\
     &=& \mathbb{E} \left[\mathbb{I}[X_3 \ne X_4]  (v_{i3} g_4 - v_{i4} g_3) (v_{j3} g_4 - v_{j4} g_3)  \Phi^\ell_p(u, X_3, v)  \Phi^\ell_p(u, X_4, v)
     \right] \\
     &=& \mathbb{E} \left[  (v_{i3} g_4 - v_{i4} g_3) (v_{j3} g_4 - v_{j4} g_3)  \Phi^\ell_p(u, X_3, v)  \Phi^\ell_p(u, X_4, v)
     \right] .
   \end{eqnarray*} 
     Now we exploit the independence of $X_3$ and $X_4$ to obtain 
      \begin{eqnarray*}
2 (R(u,v))_{i,j} 
  &=& 2 \mathbb{E} \left[  v_{i3} v_{j3} \Phi^\ell_p(u, X_3, v)  \right] 
        \mathbb{E} \left[  g_4^2 \Phi^\ell_p(u, X_4, v)  \right] - 2 \mathbb{E} \left[ v_{i3} g_3   \Phi^\ell_p(u, X_3, v) \right]
        \mathbb{E} \left[ v_{j4} g_4   \Phi^\ell_p(u, X_4, v) \right].
   \end{eqnarray*} 
  The assertion follows by dividing by 2 and re-arranging the equation.

\end{proof}

\begin{proof}[Proof of Theorem \ref{thm:var-}]

{First by direct verification we note that the following recursion for $\Phi_n^{\pmb{\ell}}$ holds. 
Starting from 
  $\Phi_{1}^{\pmb{\ell}}(x_1,x_3, x_4, x_2) = \Phi^{\ell_1}_p(x_1, x_3, x_4,
  x_2)$ we have for $n \ge 2$
\begin{align}
 &      \Phi_{n}^{\pmb{\ell}}(x_1, x_3, \ldots, x_{2n-1}, x_{2n+1},
   x_{2n+2}, x_{2n}, \ldots, x_2) \nonumber \\
&  \qquad =
   \Phi^{\ell_{n}}_p(x_{2n-1}, x_{2n+1},
   x_{2n+2}, x_{2n})
    \Phi_{n-1}^{\pmb{\ell}}(x_1, x_3, \ldots, x_{2n-1}, x_{2n}, \ldots,
   x_2)  \label{eq:38}
\end{align}
for any sequence $(x_j)_{j\ge1}$.  We abbreviate
\begin{align}
 \Phi_{n,1}^{\pmb{\ell}}(x_1, x_3, \ldots, x_{2n-1}, x, x_{2n}, \ldots,
   x_2)
&  =
  \Phi^{\ell_{n}}_p(x_{2n-1}, x, x_{2n})
    \Phi_{n-1}^{\pmb{\ell}}(x_1, x_3, \ldots, x_{2n-1}, x_{2n}, \ldots, x_2)\label{eq:50}.
\end{align}
}
The proof uses induction in $n$. First consider $n=1$. {Let $X_1, X_2, X_3,X_4$ be independent copies of $X$.} Starting from
  \eqref{eq:27}, 
  \begin{eqnarray*}
  \mathrm{Cov} \left[
\mathbf{f}(X)  \right]
&&=  \E [ (\mathbf{f}(X_2) - \mathbf{f}(X_1))(\mathbf{f}(X_2) -
    \mathbf{f}(X_1)) '  \mathbb{I}[X_1 < X_2]    ]\\ 
&&= \E \left[  \mathbb{E}\left[   \Phi^{\ell_1}_p(X_1, X_3, X_2)
    \Delta^{-\ell_1}\mathbf{f}(X_3)  \,| \, X_1, X_2 \right]
 \mathbb{E}\left[   \Phi^{\ell_1}_p(X_1, X_4, X_2)
    \Delta^{-\ell_1}\mathbf{f}(X_4) \,| \, X_1, X_2 \right]' \mathbb{I}[X_1 < X_2]    
  \right] 
  \end{eqnarray*} 
where we used \eqref{eq:28} in the last step. Now  for any $h_1$ such that $\mathbb{P}[\Delta^{-\ell_1}h_1(X)>0]=1$, dividing and multiplying by $ \sqrt{\Delta^{-\ell_1}h_1(X)}$ and applying Lemma \ref{lma:matrixcauchyschwarz} (Lagrange identity) with 
\begin{equation}\label{vg} 
\mathbf{v} (x) = \frac{ \Delta^{-\ell_1}\mathbf{f}(x)}{\sqrt{\Delta^{-\ell_1}h_1(x)}} \quad \mbox{ and } 
g(x) = \sqrt{\Delta^{-\ell_1}h_1(x)}
\end{equation}  
gives {note re-arrangement} 
 \begin{eqnarray}
\lefteqn{\mathrm{Cov} \left[ \mathbf{f}(X)  \right]  + \E \left[ R^{\ell_1}(X_1, X_2; \mathbf{v}, g)  \mathbb{I}[X_1 < X_2]   \right] } \nonumber\\
&=&  \E \left[  \mathbb{E} \left[ \mathbf{v}(X)\mathbf{v}'(X)
      \Phi_p^{\ell_1}(X_1,X, X_2)   \,| \, X_1, X_2\right]  
      \mathbb{E} \left[ g^2(X) \Phi_p^{\ell_1}(X_1, X, X_2) \,| \, X_1, X_2 \right]  \mathbb{I}[X_1 < X_2]   \right] \nonumber  \\
&=& 
  \E \left[  \mathbb{E} \left[ 
  \frac{\Delta^{-\ell_1}\mathbf{f}(X)\Delta^{-\ell_1}\mathbf{f}'(X)}{ \Delta^{-\ell_1}h_1(X)} \Phi_p^{\ell_1}(X_1,X, X_2) | X_1, X_2  \right]  
  \mathbb{E} \left[\Delta^{-\ell_1}h_1(X) \Phi_p^{\ell_1}(X_1, X, X_2) | X_1, X_2  \right] \mathbb{I}[X_1 < X_2]   \right] \nonumber \\
&=& 
 \E \left[  \mathbb{E} \left[ 
      \frac{\Delta^{-\ell_1}\mathbf{f}(X) \Delta^{-\ell_1}\mathbf{f}'(X) }{ \Delta^{-\ell_1}h_1(X)}  \Phi_p^{\ell_1}(X_1,X, X_2) | X_1, X_2 \right]  
      (h_1(X_2) - h(X_1)  ) \mathbb{I}[X_1 < X_2]   \right]\label{expand}
     \end{eqnarray} 
with the last equality following from \eqref{eq:28}. Note that, in the discrete case, the strict inequality in the indicator
$\mathbb{I}{[X_1<X_2]}$ is implicit in
$\Phi^{\ell_1}_p(X_1, X, X_2) = \chi^{\ell_1}(X_1, X)\chi^{-\ell_1}(X,
X_2)/p(X)$ (and hence a fortiori also in
$\Phi^{\ell_1}_p(X_1, X_3, X_4, X_2)$; in the continuous case there is
no difference between $\mathbb{I}{[X_1<X_2]}$ and
$\mathbb{I}{[X_1\le X_2]}$.
Hence unconditioning yields
      \begin{eqnarray*}
\lefteqn{
\E \left[  
\frac{\Delta^{-\ell_1}\mathbf{f}(X) \Delta^{-\ell_1}\mathbf{f}'(X) }{ \Delta^{-\ell_1}h_1(X)} \Phi_p^{\ell_1}(X_1,X, X_2) (h_1(X_2) - h_1(X_1)  ) \mathbb{I}[X_1 < X_2]   \right]}\\
 &=& 
\E \left[  
\frac{\Delta^{-\ell_1}\mathbf{f}(X) \Delta^{-\ell_1}\mathbf{f}'(X) }{ \Delta^{-\ell_1}h_1(X)} \Phi_p^{\ell_1}(X_1,X, X_2) (h_1(X_2) - h_1(X_1)  )  \right]\\
 &=& 
 \E \left[  
 \Delta^{-\ell_1}\mathbf{f}(X) \Delta^{-\ell_1}\mathbf{f}'(X)
   \frac{\Gamma_1^{\ell_1} h_1(X) }{ \Delta^{-\ell_1}h_1(X)} \right],
\end{eqnarray*}  
giving the first term in the covariance expansion   \eqref{eq:varianceexpansion}. 
With the notation \eqref{vg}, the  {remainder} term in \eqref{expand} is
\begin{eqnarray*}
\lefteqn{  \E \left[R^{\ell_1}(X_1, X_2; \mathbf{v}, g) \mathbb{I}[X_1 < X_2]  \right] }\\
&=&   
\E \left[ 
\mathbb{E} \left[ (\mathbf{v}_3 g_4 - \mathbf{v}_4 g_3)  (\mathbf{v}_3 g_4 - \mathbf{v}_4 g_3)' \Phi_p^{\ell_1}(X_1, X_3, X_4, X_2)  | X_1, X_2 \right]
\mathbb{I}[X_1 < X_2] \right] .  
\end{eqnarray*} 
 Now, 
\begin{eqnarray*}
 \mathbf{v}_3 g_4  =  \frac{ \Delta^{-\ell_1}\mathbf{f}(X_3)}{\sqrt{\Delta^{-\ell_1}h_1(X_3)}} \sqrt{\Delta^{-\ell_1}h_1(X_4 )} 
  =  \frac{ \Delta^{-\ell_1}\mathbf{f}(X_3)}{{\Delta^{-\ell_1}h_1(X_3)}} \sqrt{{\Delta^{-\ell_1}h_1(X_3)}\Delta^{-\ell_1}h_1(X_4 )}
\end{eqnarray*} 
and $\sqrt{{\Delta^{-\ell_1}h_1(X_3)}\Delta^{-\ell_1}h_1(X_4 )}$ is a common factor, so that 
       \begin{eqnarray*}
\lefteqn{  \E \left[R^{\ell_1}(X_1, X_2; \mathbf{v}, g) \mathbb{I}[X_1 < X_2]  \right] }\\
&=&   
\E \left[  \left(\frac{ \Delta^{-\ell_1}\mathbf{f}(X_3)}{{\Delta^{-\ell_1}h_1(X_3)}} - \frac{ \Delta^{-\ell_1}\mathbf{f}(X_4)}{{\Delta^{-\ell_1}h_1(X_4)}}\right) 
     \left(\frac{ \Delta^{-\ell_1}\mathbf{f}(X_3)}{{\Delta^{-\ell_1}h_1(X_3)}} - \frac{ \Delta^{-\ell_1}\mathbf{f}(X_4)}{{\Delta^{-\ell_1}h_1(X_4)}}\right)'   \right. \\
    && \quad \left. \times  
     \left( \sqrt{{\Delta^{-\ell_1}h_1(X_3)}\Delta^{-\ell_1}h_1(X_4 )}\right)^2
      \Phi_p^{\ell_1}(X_1, X_3, X_4, X_2) \mathbb{I}[X_1 < X_2]
\right] \\
&=& \E \left[   (\mathbf{f}_{1}(X_3) -  \mathbf{f}_{1}(X_4))  (\mathbf{f}_{1}(X_3) -  \mathbf{f}_{1}(X_4))' 
  \Delta^{-\ell_1}h_1(X_3) \Delta^{-\ell_1}h_1(X_4 ) 
   \Phi_p^{\ell_1}(X_1, X_3, X_4, X_2)
\right]   \\
&=& R_1^{\ell_1}(\mathbf{h})  
     \end{eqnarray*} 
    as required; here $\mathbf{h} = h_1$.   Thus the assertion holds for $n=1$.

 To obtain the complete claim, we proceed by induction and suppose
 that the claim holds at some $n$. It  remains to show that
 \begin{align}
   \label{eq:88}
   R_n^{\pmb \ell}(\mathbf{h}) = 
   \mathbb{E} \left[ \Delta^{-\ell_{n+1}} \mathbf{f}_n(X) 
   \Delta^{-\ell_{n+1}}   \mathbf{f}'_n(X)
   \frac{\Gamma_{n+1}^{\pmb\ell}\mathbf{h}(X)}{\Delta^{-\ell_{n+1}}h_{n+1}(X)}
   \right]  - R_{n+1}^{\pmb \ell}(\mathbf{h}).
 \end{align}
To this purpose, starting from
 \eqref{eq:89}, we simply apply the same process as above: for $x_{2n+1}<x_{2n+2}$, we use 
 \begin{align*}
 \mathbf{f}_n(x_{2n+2}) - \mathbf{f}_n(x_{2n+1}) 
 & =  \mathbb{E}\left[ \Delta^{-\ell_{n+1}} \mathbf{f}_n(X) 
    \Phi^{\ell_{n+1}}_p(x_{2n+1}, X, x_{2n+2}) \right] 
 \end{align*}
 as well as the Lagrange identity  \eqref{eq:75} and simple
 conditioning to obtain that 
 \begin{eqnarray*}
  R_n^{\pmb \ell}(\mathbf{h})
  & = & 
  \E \Bigg[ \left(\mathbf{f}_n(X_{2n+2})- \mathbf{f}_n(X_{2n+1})\right)
   \left(\mathbf{f}_n(X_{2n+2})- \mathbf{f}_n(X_{2n+1})\right)'                         \\
& &
\quad \quad \Phi_{n}^{\pmb{\ell}}(X_1, \ldots X_{2n+1},  X_{2n+2}, \ldots,  X_2) \prod_{i=1}^n   \Delta^{-\ell_{i}} h_i(X_{2i+1}, X_{2i+2}) \Bigg] \\
 &=& 
  \E \Bigg[ 
  \mathbb{E} \left[ \Delta^{-\ell_{n+1}} \mathbf{f}_n(X_{2n+3}) 
    \Phi^{\ell_{n+1}}_p(X_{2n+1}, X_{2n+3}, X_{2n+2}) |  X_{2n+1},  X_{2n+2}\right] 
      \\
&& 
\quad   \quad   \mathbb{E}\left[\Delta^{-\ell_{n+1}} \mathbf{f}'_n(X_{2n+4})
    \Phi^{\ell_{n+1}}_p(X_{2n+1}, X_{2n+4}, X_{2n+2}) |  X_{2n+1},  X_{2n+2}\right] 
\\
  & &
\quad \quad \Phi_{n}^{\pmb{\ell}}(X_1, \ldots X_{2n+1},  X_{2n+2}, \ldots,  X_2) \prod_{i=1}^n   \Delta^{-\ell_{i}} h_i(X_{2i+1}, X_{2i+2})  \Bigg] .
 \end{eqnarray*}  
 Now  for any $h_{n+1}$ such that $\mathbb{P}[\Delta^{-\ell_{n+1}}h_{n+1}(X)>0]=1$, dividing and multiplying by $ \sqrt{\Delta^{-\ell_{n+1}}h_{n+1}(X)}$ and applying Lemma \ref{lma:matrixcauchyschwarz} with 
\begin{equation}\label{vgn} 
\mathbf{v}_{n+1} (x) = \frac{ \Delta^{-\ell_{n+1}}\mathbf{f}_n(x)}{\sqrt{\Delta^{-\ell_{n+1}}h_{n+1}(x)}} \quad \mbox{ and } g_{n+1}(x) = \sqrt{\Delta^{-\ell_{n+1}}h_{n+1}(x)}
\end{equation}  
we obtain  with \eqref{eq:2} 
 \begin{eqnarray} \label{expand2} 
  \lefteqn{  R_n^{\pmb \ell}(\mathbf{h}) - \E\bigg[ \E \left[ R^{\ell_{n+1}}(X_{2n+1}, X_{2n+2}; \mathbf{v_{n+1}}, g_{n+1})  |  X_{2n+1},  X_{2n+2}  \right]  } \nonumber \\
  &&\quad \quad \quad \mathbb{I}[X_{2n+1}<X_{2n+2}] \Phi_{n}^{\pmb{\ell}}(X_1, \ldots X_{2n+1},  X_{2n+2}, \ldots,  X_2) \prod_{i=1}^n   
\Delta^{-\ell_{i}} h_i(X_{2+1}, X_{2i+2})    \bigg] \nonumber \\
&=& 
\E \bigg[   \mathbb{E} \left[ \mathbf{v}_{n+1}(X)\mathbf{v}'_{n+1}(X) 
      \Phi_p^{\ell_{n+1}}(X_{2n+1},X, X_{2n+2})  |  X_{2n+1},  X_{2n+2} \right] \nonumber \\
&& 
\quad \quad  
\times \mathbb{E} \left[ g_{n+1}^2(X) \Phi_p^{\ell_{n+1}}(X_{2n+1}, X, X_{2n+2})  |  X_{2n+1},  X_{2n+2}\right]     \nonumber \\
&&   
\quad \quad \mathbb{I}[X_{2n+1}<X_{2n+2}] \Phi_{n}^{\pmb{\ell}}(X_1, \ldots X_{2n+1},  X_{2n+2}, \ldots,  X_2) \prod_{i=1}^n   
\Delta^{-\ell_{i}} h_i(X_{2+1}, X_{2i+2})    \bigg] \nonumber \\
 &=& 
\E \Big[   
\mathbb{E} \left[ \mathbf{v}_{n+1}(X)\mathbf{v}'_{n+1}(X) 
 \Phi_p^{\ell_{n+1}}(X_{2n+1},X, X_{2n+2})  \right]  
 (h_{n+1}(X_{2n+2} ) - h_{n+1}(X_{2n+1}))   \nonumber  \\
&& 
\Phi_{n}^{\pmb{\ell}}(X_1, \ldots X_{2n+1},  X_{2n+2}, \ldots,  X_2) 
\prod_{i=1}^n  \Delta^{-\ell_{i}} h_i(X_{2i+1}, X_{2i+2})  \Big] \nonumber \\
&=&  
\mathbb{E} \left[
 \Delta^{-\ell_{n+1}} \mathbf{f}_n(X)  
 \Delta^{-\ell_{n+1}} \mathbf{f}'_n(X)
 \frac{\Gamma_{n+1}^{\pmb\ell}\mathbf{h}(X)}{\Delta^{-\ell_{n+1}}h_{n+1}(X)}  \right] 
 \end{eqnarray} 
where we used \eqref{vgn} in the last step. Thus we have recovered the first summand in \eqref{eq:88}. For the {remainder term} in \eqref{expand2}, leaving out the negative sign, the notation \eqref{vgn} gives  
 \begin{eqnarray*} 
 \lefteqn{\E \big[  
  \E \left[ R^{\ell_{n+1}}(X_{2n+1}, X_{2n+2}; \mathbf{v_{n+1}}, g_{n+1})  |  X_{2n+1},  X_{2n+2}  \right]  
  }
 \\
 && 
  \quad \mathbb{I}[X_{2n+1} < X_{2n+2}]  \Phi_{n}^{\pmb{\ell}}(X_1, \ldots X_{2n+1},  X_{2n+2}, \ldots,  X_2) \prod_{i=1}^n   \Delta^{-\ell_{i}} h_i(X_{2i+1}, X_{2i+2})   \big]  \\
&=& 
 \E \Big[ 
(\mathbf{v}_{n+1, 2n+3} g_{n+1,2n+4} - \mathbf{v}_{n+1, 2n+4} g_{{n+1}, 2n+3})
 (\mathbf{v}_{n+1, 2n+3} g_{n+1,2n+4} - \mathbf{v}_{n+1, 2n+4} g_{{n+1}, 2n+3})'  \\
&& 
   \Phi_p^{\ell_{n+1}}(X_{2n+1}, X_{2n+3}, X_{2n+4}, X_{2n+2})
 \Phi_{n}^{\pmb{\ell}}(X_1, \ldots X_{2n+1},  X_{2n+2}, \ldots,  X_2) \prod_{i=1}^n   \Delta^{-\ell_{i}} h_i(X_{2i+1}, X_{2i+2})   \Big] 
  \end{eqnarray*} 
Again extracting the common factor $\sqrt{{\Delta^{-\ell_{n+1}}h_{n+1}(X_{2n+3})}\Delta^{-\ell_{n+1}}h_{n+1}(X_{2n+4} )}$ 
and re-arranging yields the assertion. 

\end{proof}

\begin{proof}[Proof of Lemma \ref{lma:gamma2cont}]
  Let $x_1 \le x \le x_2$ and $h$ an increasing function.  Direct
  application of the definitions {with \eqref{eq:phiex}} lead to
  \begin{align*}
  p(x) \gamma_k^0 h(x_1, x, x_2) & = \int_{x_1}^x  \int_x^{x_2}
                              \int_{x_3}^x \int_x^{x_4} \cdots
    \int_{x_{2k-3}}^x  \int_x^{x_{2k-2}}
(h(x_{2k})- h(x_{2k-1})h'(x_{2k-1})h'(x_{2k}) \mathrm{d}x_{2k} \mathrm{d}x_{2k-1}\\
    & 
\qquad \qquad     \cdots h'(x_{5})h'(x_{6})\mathrm{d}x_{6} \mathrm{d}x_{5}h'(x_{3})h'(x_{4})\mathrm{d}x_{4} \mathrm{d}x_{3}.
  \end{align*}
  Applying the change of variables $u_k = h(x_k), k=1, \ldots, 2k$ and
  setting $u = h(x)$ we see that the sequence $\gamma_k^0 h$ depends
  only on the iterated integrals
  \begin{align*}
 \iota_k(u_1, u, u_2) & := \int_{u_1}^{u}\int_{u}^{u_2}
                               \int_{u_3}^{u}   \int_{u}^{u_4}  \cdots
    \int_{u_{2k-3}}^u\int_{u}^{u_{2k-2}}
                              (u_{2k}- u_{2k-1})\mathrm{d}u_{2k} \mathrm{d}u_{2k-1}
                               \cdots \mathrm{d}u_{6}
                               \mathrm{d}u_{5}
                               \mathrm{d}u_{4}
                                 \mathrm{d}u_{3}
  \end{align*}
  which we can write recursively as 
  \begin{align*}
 \iota_1(u_1, u, u_2) & = u_2-u_1\\
     \iota_{k}(u_1, u, u_2) & = \int_{u_1}^{u}\int_{u}^{u_2}
                                   \iota_{k-1}(u_3, u, u_4)
                                   \mathrm{d}u_4 \mathrm{d}u_3, \qquad  k \ge
                                   {2}. 
  \end{align*}
  It remains to show that
  \begin{equation}\label{eq:11}
                 \iota_{k}(u_1, u, u_2) = (u_2-u)^{k-1}(u -
                                  u_1)^{k-1}(u_2 - u_1)  \frac{\mathbb{I}[u_1\le u \le
                                u_2]}{k!(k-1)!}                 
                            \end{equation}
                            for all $k \ge 1$.  We proceed by
                            induction on $k$.  Clearly
                            $\iota_1(u_1, u, u_2) =
                            (u_2-u_1)\mathbb{I}[u_1\le u \le u_2]$, as
                            required. Next suppose that \eqref{eq:11}
                            holds.  Then
\begin{align*}
           \iota_{k+1}(u_1, u, u_2) 
  & = \frac{1}{k!(k-1)!}\int_{u_1}^{u}\int_{u}^{u_2}(u_4-u)^{k-1}(u -
                                  u_3)^{k-1}(u_4 - u_3)  \mathrm{d}u_4
    \mathrm{d}u_3 \\
  & = \frac{1}{k!(k-1)!}\int_{u_1}^{u}\int_{u}^{u_2}(u_4-u)^{k}(u -
                                  u_3)^{k-1}  \mathrm{d}u_4
    \mathrm{d}u_3  \\
  & \qquad + \frac{1}{k!(k-1)!}\int_{u_1}^{u}\int_{u}^{u_2}(u_4-u)^{k}(u -
                                  u_3)^{k-1}  \mathrm{d}u_4
    \mathrm{d}u_3 \\
  & = \frac{(u_2-u)^{k+1}(u-u_1)^k + (u_2-u)^{k}(u-u_1)^{k+1}}{(k+1)!k!}
\end{align*}
which leads to the claim. 
 \end{proof}

 \begin{proof}[Proof of Identity \eqref{eq:57cont2}]
Identity \eqref{eq:57cont2} follows 
from Lemma \ref{lma:gamma2cont} by using 
$h(X_2) - h(X_1) = h(X_2) - h(x) + h(x) - h(X_1)$ 
{and $\mathbb{I}[X_1 \le x \le X_2] \mathbb{I}[X_1 \ne X_2]= \mathbb{I}[X_1\le x] \mathbb{I}[X_2 \ge x] \mathbb{I}[X_1 \ne X_2]$} to get
\begin{align}
\Gamma_{k}^{\pmb 0}h(x) 
&= (-1)^{k-1}\frac{1}{p(x)}\mathbb{E}\big[H_x^{k-1}(X)\mathbb{I}[X\le x] \big]  
                                                                 \mathbb{E}\big[H_x^k(X)\mathbb{I}[X\ge x] \big] \nonumber \\
  & \quad +  (-1)^k\frac{1}{p(x)} \mathbb{E}\big[H_x^k(X)\mathbb{I}[X\le x] \big]  
\mathbb{E}\big[H_x^{k-1}(X)\mathbb{I}[X\ge x] \big] \label{eq:gammatrick} \\
&= (-1)^{k-1}\mathbb{E}\big[H_x^{k-1}(X)\big]  \frac{1}{p(x)}
\mathbb{E}\big[H_x^k(X)\mathbb{I}[X\ge x] \big] 
+  (-1)^k \mathbb{E}\big[H_x^k(X) \big]  
\frac{1}{p(x)}\mathbb{E}\big[H_x^{k-1}(X)\mathbb{I}[X\ge x] \big] \nonumber
\end{align}
where the last equality follows from
\begin{equation*}
\mathbb{E}\big[H_x^k(X) \big] = \mathbb{E}\big[H_x^k(X)\mathbb{I}[X\le x]  \big] +\mathbb{E}\big[H_x^k(X)\mathbb{I}[X\ge x]  \big].
\end{equation*}
Upon noting that
{
\begin{eqnarray*}  &&  - \mathcal{L}_p^0 H_x^k(x) \\ &&=\frac{1}{p(x)} \left\{ \mathbb{E}\big[H_x^k(X_2)  \mathbb{I}[X_1 < x < X_2]\big] -  \mathbb{E}\big[H_x^k(X_1)  \mathbb{I}[X_1 < x < X_2]\big] \right\}\\
&&= \frac{1}{p(x)} \left\{ \mathbb{E}\big[H_x^k(X_2)  \mathbb{I}[x < X_2]\big] \mathbb{P}[x > X_1] -  \mathbb{E}\big[H_x^k(X_1)  \mathbb{I}[X_1 < x \big] \mathbb{P}[x < X_2] \right\}
\\
&&= \frac{1}{p(x)} \left\{ \mathbb{E}\big[H_x^k(X_2)  \mathbb{I}[x < X_2]\big]  -  \mathbb{E}\big[H_x^k(X_2)  \mathbb{I}[x < X_2]\big] \mathbb{P}[x < X_1] -  \mathbb{E}\big[H_x^k(X_1)  \mathbb{I}[X_1 < x \big] \mathbb{P}[x < X_2] \right\}
\end{eqnarray*}
wtih $P(x) = \mathbb{P}[X \le x]$ we obtain}
\begin{equation*}
\frac{1}{p(x)} \mathbb{E}\big[H_x^k(X) \mathbb{I}[X\ge x]\big] 
= - \mathcal{L}_p^0 H_x^k(x) 
+   \frac{1-P(x)}{p(x)} \mathbb{E}\big[H_x^k(X) \big], 
\end{equation*}
the required result is obtained after straightforward simplifications {by writing 
\begin{align*}
\Gamma_{k}^{\pmb 0}h(x) 
&= (-1)^{k-1}\mathbb{E}\big[H_x^{k-1}(X)\big]  \frac{1}{p(x)}
\mathbb{E}\big[H_x^k(X)\mathbb{I}[X\ge x] \big] 
+  (-1)^k \mathbb{E}\big[H_x^k(X) \big]  
\frac{1}{p(x)}\mathbb{E}\big[H_x^{k-1}(X)\mathbb{I}[X\ge x] \big] \\
&= (-1)^{k-1}\left( -\mathbb{E}\big[H_x^{k-1}(X)\big]  
 \mathcal{L}_p^0 H_x^k(x) + \mathbb{E}\big[H_x^k(X) \big]   \mathcal{L}_p^0 H_x^{k-1}(x)  \right) \\
&
+ (-1)^{k-1}\frac{1-P(x)}{p(x)}   \left( \mathbb{E}\big[H_x^{k-1}(X)\big]  
\mathbb{E}\big[H_x^k(X) \big]  
-  \mathbb{E}\big[H_x^k(X) \big] \mathbb{E}\big[H_x^{k-1}(X) \big] \right) 
\end{align*}
and noticing that the last term cancels. 
} 
\end{proof}

\begin{proof}[{Proof of Lemma \ref{lma:identitycontdisc}}]
 
  We shall prove that
   \begin{align}
     & \gamma_{k}^{\pmb \ell}  (x_1, x, x_2)  :=  \gamma_{k}^{\pmb \ell} \mathrm{Id}(x_1, x, x_2) 
     = 
(x_{2}-x)_{\{k-1;\pmb\ell\}}  (x-x_{1})^{\{k-1;\pmb\ell\}}(x_{{2}}-x_{1})
\frac{ \mathbb{I}[x_1 + \mathbf{a}_{k}  \le x \le x_2 -
       \mathbf{b}_{k}]}{p(x) k!(k-1)!}.
       \label{eq:57Discrete}
   \end{align}
   The claim is obvious from \eqref{eq:57cont} in the continuous
case. For the discrete case,  the assertion is proved by induction in $k$; the cases $k=1$ and
  $k=2$ need to be asserted to start the induction.  The case $k=1$ is
  immediate.  
  For $k=2$, we show that
  $$\gamma_2^{\ell_1, \ell_2}(X_1,x,X_2) 
  = \frac{1}{2}(x-X_1-a_{\pmb\ell}{(2)}+1)(X_2-x-b_{\pmb\ell}{(2)}+1)(X_2-X_1) \frac{\mathbb{I}[X_1{+a_{\pmb \ell}(2)}\leq x \leq X_2{-b_{\pmb \ell}(2)}]}{p(x)}$$ for $\ell_i\in\{-1,1\}$. 
To this end, from {Proposition} \ref{lma:gamma2} {where we sum over $(x_3,x_4)$ instead of $(y,z)$}, we obtain 
\begin{align*}
\gamma_{2}^{\ell_1,\ell_2}(x_1,x,x_2)
&= \sum_{x_3=x_1+a_1}^{x-a_2}\sum_{x_4=x+b_2}^{x_2-b_1} (x_4-x_3)  \frac{\mathbb{I}[x_1+\mathbf{a}_{2}\leq x\leq x_2-\mathbf{b}_2]}{p(x)} \\
&= \frac{1}{2} (x-x_1-\mathbf{a}_2+1)(x_2-x-\mathbf{b}_2+1)(x_2-x_1) \frac{\mathbb{I}[x_1+\mathbf{a}_{2}\leq x\leq x_2-\mathbf{b}_2]}{p(x)}
\end{align*}
 as required.

To conclude the argument, we prove the identity \eqref{eq:57Discrete} by induction: we suppose the claims hold for $k$ and investigate its validity
for $k+1$. The definition of $\Gamma_k^{\pmb \ell}$ in \eqref{eq:2} gives
\begin{align}
\gamma_{k+1}^{\pmb \ell}(x_1,x,x_2)
&= 
\mathbb{E} \bigg[
\frac{\chi^{\ell_1}(x_{1},X_{3})}{p(X_{3})} \frac{\chi^{-\ell_1}(X_{4},x_{2})}{p(X_{4})}
\gamma_{k}^{\ell_2,\ldots,\ell_{k+1}}(X_3,x,X_4) \bigg] \label{eq:gammak1}
\end{align}
Now we can plug-in the induction assumption \eqref{eq:57Discrete} into \eqref{eq:gammak1}:
\begin{eqnarray*}
  \lefteqn{\gamma_{k+1}^{\pmb \ell}(x_1,x,x_2)} \\
  &=& \E\bigg[
      (X_{4}-x-\mathbf{b}'_{k}+1)^{[k-1]}  (x-X_{3}-\mathbf{a}'_k+1)^{[k-1]}(X_4-X_3) \frac{ \mathbb{I}[X_3 + \mathbf{a}'_{k}  \le x \le X_4 - \mathbf{b}'_{k}]}{p(x) k!(k-1)!} \\
      &&\quad \frac{\chi^{\ell_1}(x_{1},X_{3})}{p(X_{3})} \frac{\chi^{-\ell_1}(X_{4},x_{2})}{p(X_{4})}\bigg] \\
  &=& \sum_{x_3=x_1+a_1}^{x-\mathbf{a}'_k}
      \sum_{x_4=x+\mathbf{b}'_k}^{x_2-b_1} 
      (x_{4}-x-\mathbf{b}'_{k}+1)^{[k-1]}(x-x_{3}-\mathbf{a}'_k+1)^{[k-1]}(x_4-x_3) 
      \frac{\mathbb{I}[x_1+\mathbf{a}_{k+1}\leq x\leq x_2-\mathbf{b}_{k+1}]}{p(x)} \\
  &=& (x_2-x-\mathbf{b}_{k+1}+1)^{[k]}(x-x_1-\mathbf{a}_{k+1}+1)^{[k]}(x_2-x_1) 
      \frac{\mathbb{I}[x_1+\mathbf{a}_{k+1}\leq x\leq x_2-\mathbf{b}_{k+1}]}{p(x)}   
\end{eqnarray*}

where $\mathbf{a}'_{k}=\sum_{i=2}^{k+1}a_i$ and  $\mathbf{b}'_{k}=\sum_{i=2}^{k+1}b_i$.
\end{proof}

\begin{proof}[Proof of Lemma \ref{lma:cdf}] 
By Lemma \ref{lma:gamma2cont} {and \eqref{eq:gammatrick}}, we have
\begin{align*}
\Gamma_k^0P(x) 
=& \frac{1}{p(x)k!(k-1)!}\E\left[ (P(x)-P(X_1))^{k-1}\mathbb{I}[X_1\leq x] \right]
\E\left[ (P(X_2)-P(x))^{k}\mathbb{I}[X_2\geq x] \right]  \\
& + \frac{1}{p(x)k!(k-1)!}\E\left[ (P(x)-P(X_1))^{k}\mathbb{I}[X_1\leq x] \right]
\E\left[ (P(X_2)-P(x))^{k-1}\mathbb{I}[X_2\geq x] \right].
\end{align*}
Moreover, {using integration by substitution},
\begin{align*}
\E\left[ (P(x)-P(X_1))^{k}\mathbb{I}[X_1\leq x] \right]
&{= \int_a^x (P(x)-P(x_1))^{k} p(x_1) dx_1 
= -\int_{P(x)}^0 u^k du}
= \frac{P(x)^{k+1}}{k+1} \\
\E\left[ (P(X_2)-P(x))^{k}\mathbb{I}[X_2\geq x] \right]
&{= \int_x^b (P(x_2)-P(x))^{k} p(x_2) dx_2 
= \int_0^{1-P(x)} u^k du}
= \frac{(1-P(x))^{k+1}}{k+1},
\end{align*}
and the conclusion follows. 
\end{proof}

\medskip
\begin{proof}[Proof of Proposition \ref{prop:perason}] 
The argument for the integrated Pearson system is inspired from \cite[Theorem 2]{johnson1993note}.  
By Lemma \ref{lma:gamma2cont}, note that 
\begin{align*}
\gamma_k^0(x_1,x,x_2) &= 
(x-x_{1})^{k-1}(x_{2}-x)^{k-1}(x_{2}-x_{1}) \frac{\mathbb{I} [x_1 \le x \le x_2]}{p(x)k!(k-1)!} \\
& = (x-x_{1})^{k-1}(x_{2}-x)^{k-1}(x_{{2}}-\mu+\mu -x_{1}) \frac{\mathbb{I} [x_1 \le x]\mathbb{I}[x\le x_2]}{p(x)k!(k-1)!} 
\end{align*}
Therefore, $\Gamma_k^0(x)$ can be decomposed using simple expectations: 
\begin{align}
\Gamma_k^0(x) =&\frac{1}{p(x)k!(k-1)!}\Bigg( \mathbb{E}\left[(x-X_{1})^{k-1}\mathbb{I} [X_1 \le x]\right]\mathbb{E}\left[(X_{{2}}-\mu)(X_{2}-x)^{k-1}\mathbb{I}[x\le X_2]\right] \nonumber\\
&+ \mathbb{E}\left[(\mu-X_{1})(x-X_{1})^{k-1}\mathbb{I} [X_1 \le x]\right]\mathbb{E}\left[(X_{2}-x)^{k-1}\mathbb{I}[x\le X_2]\right] \Bigg) \label{eq:gammaP}
\end{align}
{In the continuous setting, the Stein kernel $\tau_p$ is such that is satisfies for $X \sim p$ with mean $\mu$ and differentiable $f$ such that the expectations exist, 
$$ \mathbb{E} [ (X - \mu) f(X) ]= \mathbb{E} [ \tau_p(X) f'(X)].$$}
Integrating by parts we {thus} obtain
\begin{align*}
\mathbb{E} \left[(X_2-\mu) (X_2-x)^{k-1} \mathbb{I} [X_2 \ge x ]\right]
&= \mathbb{E} \left[ \tau_p(X_2)(k-1)(X_2-x)^{k-2} \mathbb{I} [X_2 \ge x ]\right] 
\end{align*}
and
\begin{align*}
\mathbb{E} \left[(\mu-X_1) (x-X_1)^{k-1} \mathbb{I} [X_1 \le x ]\right]
&= \mathbb{E} \left[ \tau_p(X_1)(k-1)(x-X_1)^{k-2} \mathbb{I} [X_1 \le x ]\right].
\end{align*}
When we plug it into \eqref{eq:gammaP}, we get 
\begin{align*}
\Gamma_k^0(x) &= \frac{k-1}{p(x)k!(k-1)!} 
\mathbb{E} \Bigg[ (x-X_{1})^{k-2}(X_{2}-x)^{k-2}(\tau_p(X_2)(x-X_1)+ \tau_P(X_1)(X_2-x)  {)} \mathbb{I} [X_1 \le x \le X_2]\Bigg].
\end{align*}
Using the particular form of $\tau_p$ for the integrated Pearson family, Taylor expansion of $\tau_p(X) $ around $x$ gives 
\begin{align*}
(x-x_1)\tau_p(x_2)+(x_2-x)\tau_p(x_1) = 
\tau_p(x)(x_2-x_1) + \frac{\tau_p''(x)}{2}(x-x_1)(x_2-x)(x_2-x_1)
\end{align*}
Therefore, 
\begin{align*}
\Gamma_k^0(x)
=&  \frac{k-1}{k!(k-1)!} \frac{1}{p(x)} 
\mathbb{E} \Big[ (x-X_{1})^{k-2} (X_2-x)^{k-2} \mathbb{I} [X_1 \leq x \leq X_2 ]\\
&\quad \left(\tau_p(x)(X_2-X_1) + \frac{\tau_p''(x)}{2}(x-X_1)(X_2-x)(X_2-X_1)\right)
\Big]   \\
=& \frac{\tau_p(x)}{k} \Gamma_{k-1}^0(x) + \frac{\tau_p''(x)(k-1)}{2} \Gamma_k^0(x) \\
=& \frac{1}{k\left(1-\frac{k-1}{2}\tau_p''(x)\right)} \tau_p(x) \Gamma_{k-1}^0(x)
\end{align*}
The assertion follows from {iterating this expression and using} $\Gamma_1^0(x)=\tau_p(x)$ and $\tau_p''(x)=2\delta$. 
\end{proof}

\begin{proof}[Proof of Proposition \ref{prop:perasondisc}] 
By induction, we only have to prove
the relation with respect to $\ell_{k+1}$, i.e.,
\begin{equation}
\Gamma_{k+1}^{\pmb \ell,1} (x) = \frac{\tau_p^+(x-\mathbf{a}_{k})}{(k+1)(1- k \delta)} \Gamma_{k}^{\pmb \ell} (x)  \text{ and }
\Gamma_{k+1}^{\pmb \ell,-1} (x) = \frac{\tau_p^-(x+\mathbf{b}_{k})}{(k+1)(1- k \delta)} \Gamma_{k}^{\pmb \ell} (x).
\nonumber
\end{equation}
The following argument is inspired from \cite{APP07}. 
Using \eqref{eq:57Discrete} and a similar proof as in the Pearson case (Proposition \ref{prop:perason}), we may rewrite $\Gamma_{k+1}^{\pmb\ell,1}(x)$ using simple expectations: 
\begin{align}
&\Gamma_{k+1}^{\pmb\ell,1}(x) 
= 
\frac{1}{p(x)}\frac{1}{k!(k+1)!}\Bigg(\nonumber\\
& \quad\quad \quad \quad \mathbb{E}\left[(x-X_1-\mathbf{a}_{k})^{[k]}\mathbb{I}[X_1 + \mathbf{a}_{k}+1  \le x]\right]  
\mathbb{E}\left[(X_2-\mu)(X_{2}-x-\mathbf{b}_{k} +1)^{[k]} \mathbb{I}[x \le X_2 - \mathbf{b}_{k}]\right] \nonumber\\
&\quad \quad \quad \quad + \mathbb{E}\left[(\mu-X_1)(x-X_{1}-\mathbf{a}_{k})^{[k]}\mathbb{I}[X_1 + \mathbf{a}_{k}+1  \le x]\right] \mathbb{E}\left[(X_{2}-x-\mathbf{b}_{k} +1)^{[k]} \mathbb{I}[x \le X_2 - \mathbf{b}_{k}]\right]
\Bigg). \label{eq:gammaO}
\end{align}
{With the notation \eqref{eq:discrprod1} is it straightforward to verify that  
for all
  $x$ we have}
   \begin{equation}\label{lma:diffcal}
    \Delta^\ell \left( f^{[k]}(x) \right) = f^{[k-1]}(x+a_{\ell}) \sum_{j=0}^{k-1}
    \Delta^\ell f(x+j).
  \end{equation}
In particular, for all
  $x,a$, we have 
\begin{align*}
  \Delta^-\left( (x-a+1)^{[k]} \mathbb{I}[x\geq a] \right)&  = k (x-a+1)^{[k-1]} \mathbb{I}[x\geq a]  
  \\
  \Delta^+\left( (a+1-x)^{[k]} \mathbb{I}[x\leq a] \right)&  = -k (a+1-x)^{[k-1]} \mathbb{I}[x\leq a] 
  \\
  \Delta^-\left( (a-x)^{[k]} \mathbb{I}[x< a] \right)&  = -k (a-x+1)^{[k-1]} \mathbb{I}[x\leq a] 
\end{align*}
{{The Stein kernel $\tau_p^\ell$ for discrete distributions satisfies  for $X \sim p$ with mean $\mu$ and functions $f$ such that the expectations exist, 
$$ \mathbb{E} [ (X - \mu) f(X) ]=
 \mathbb{E} [ \tau_p^\ell(X) \Delta^{-\ell} f(X-\ell)],$$
see for example \cite{ley2017stein}. 
}}
Hence, with \eqref{lma:diffcal}, we may use the {discrete integration by parts} formula to rewrite
\begin{align*}
\mathbb{E} \bigg[(X_2-\mu)(X_{2}-x-\mathbf{b}_{k}+1)^{[k]} \mathbb{I}{[x \le X_2 - \mathbf{b}_{k}]} \bigg] 
=& k \mathbb{E} \bigg[\tau_p^+(X_2) (X_{2}-x-\mathbf{b}_{k}+1)^{[k-1]}  \mathbb{I}{[x \le X_2 - \mathbf{b}_{k}]} \bigg]  
\end{align*}
and
\begin{align*}
\mathbb{E} \bigg[(\mu-X_1)(x-X_1-\mathbf{a}_{k})^{[k]}  \mathbb{I}[X_1 \le x - \mathbf{a}_{k}-1] \bigg] 
=& \mathbb{E} \bigg[(\mu-X_1)(x-X_1-\mathbf{a}_{k})^{[k]}  \mathbb{I}{[X_1 \le x - \mathbf{a}_{k}]} \bigg] \\
=& k \mathbb{E} \bigg[\tau_p^+(X_1) (x-X_1-\mathbf{a}_{k}+1)^{[k-1]}  \mathbb{I}{[X_1 \le x - \mathbf{a}_{k}]} \bigg]  .
\end{align*}
After plugging these equations into \eqref{eq:gammaO} and some further algebraic developments (which we omit), we obtain
\begin{align*}
\Gamma_{k+1}^{\pmb \ell,1} (x) 
=&\frac{1}{p(x)}\frac{1}{k!(k+1)!} \Bigg(
k \tau_p^+(x-\mathbf{a}_{k}) \\
&\mathbb{E}\bigg[ (x-X_1-\mathbf{a}_{k}+1)^{[k-1]}(X_2-x-\mathbf{b}_{k}+1)^{[k-1]}
(X_2-X_1)  \mathbb{I}{[X_1+\mathbf{a}_{k} \leq x \leq X_2 - \mathbf{b}_{k}]} \bigg] \\
&+  \delta k \mathbb{E}\bigg[ (X_2-X_1)(x-X_1-\mathbf{a}_{k})(X_2-x+k-\mathbf{b}_{k}) \mathbb{I}{[X_1+\mathbf{a}_{k}+1 \leq x \leq X_2 - \mathbf{b}_{k}]} \bigg] \Bigg) \\
=& \frac{\tau_p^+(x-\mathbf{a}_{k})}{k+1}\Gamma_{k}^{\pmb \ell} (x) 
 + \delta k \Gamma_{k+1}^{\pmb \ell,1} (x)
\end{align*}
which gives the assertion. The same result can easily be obtained for $\Gamma_{k+1}^{\pmb \ell,-1} (x) $.
\end{proof}

\end{document}